\documentclass[11pt]{article}
\usepackage{latexsym, amsmath, amsfonts, amsthm, amssymb,bbm,calrsfs}
\usepackage{a4wide}
\usepackage[shortlabels]{enumitem}
\usepackage[colorlinks=true]{hyperref}

\newcommand{\R}{\mathbb R}
\newcommand{\E}{\mathbb E}
\newcommand{\var}{\mathrm{var}}
\renewcommand{\epsilon}{\varepsilon}

\newcommand{\e}{\mathrm e}
\newcommand{\id}{\mathrm{Id}}
\renewcommand{\phi}{\varphi}
\renewcommand{\tilde}{\widetilde}
\newtheorem{theorem}{Theorem}
\newtheorem{lemma}[theorem]{Lemma}
\newtheorem{proposition}[theorem]{Proposition}
\newtheorem{corollary}[theorem]{Corollary}
\theoremstyle{remark}
\newtheorem*{remark}{Remark}

\theoremstyle{definition}

\begin{document}

\title{Convergence in total variation for the kinetic Langevin algorithm} 
\author{Joseph Lehec\footnote{Universit\'e de Poitiers, CNRS, LMA, Poitiers, France}}

\maketitle

\begin{abstract}
We prove non asymptotic total variation estimates for 
the kinetic Langevin algorithm in high dimension when the target measure 
satisfies a Poincar\'e inequality and has gradient Lipschitz potential. 
The main point is that the estimate improves significantly upon the  
corresponding bound for the non kinetic version of the algorithm, 
due to Dalalyan. In particular the dimension dependence drops 
from $O(n)$ to $O(\sqrt n)$. 
\end{abstract}

\section{Introduction}
\subsection{Context} 
Suppose we want to sample from a probability measure $\mu$ on 
$\R^n$ of the form 
\[ 
\mu (dx) = \e^{-V(x)} \,dx 
\] 
where $V$ is some smooth function from $\R^n$ to $\R$
(never mind the precise hypothesis for now) which 
we call the potential of $\mu$. 
This is a very common problem in 
applied mathematics, it shows up in many different contexts, 
from Bayesian statistics, to optimization, machine learning 
and many more. We will not discuss applications here at all.  
Instead we focus on the sampling problem from 
a theoretical point of view. We shall investigate a particular 
algorithm called the \emph{kinetic Langevin algorithm}. This is 
an order $1$ algorithm, in the sense that it relies 
on the knowledge of the gradient of $V$. 
Therefore it is assumed throughout this article that
there is some oracle that given a point $x$ in $\R^n$ 
returns $\nabla V (x)$. 
In this context the complexity of the sampling 
algorithm is the number of oracle queries. 

\subsection{The Langevin algorithm and its kinetic version}
Consider the following 
stochastic differential equation 
\[ 
d X_t = \sqrt 2 \, dW_t - \nabla V (X_t) \, dt , 
\]
where $(W_t)$ is a standard Brownian motion on $\R^n$. 
Under mild hypotheses the equation admits 
a unique strong solution, which is a Markov process, for which $\mu$ 
is the unique stationary measure. Moreover the process 
is ergodic, in the sense that we have convergence 
of $X_t$ to $\mu$ in law as $t$ tends to $+\infty$. 
The Langevin algorithm is the Markov chain induced by 
the Euler scheme associated to this diffusion. This means that given a 
time step parameter $\eta$, the algorithm is given by   
\[ 
x_{k+1} = x_k + \sqrt{2 \eta} \, \xi_{k+1} - \eta \nabla V ( x_k) , 
\] 
where $(\xi_k)$ is an i.i.d. sequence of standard Gaussian vectors on $\R^n$. 
There is quite a lot of literature studying the performance of this algorithm,
either empirically or from a more theoretical point of view. 
Explicit non asymptotic bounds seem to have started with the seminal 
work of Dalalyan~\cite{dalalyan1} to which we will come back later on. 

Let us move on to the kinetic version of the algorithm. 
The main idea is to add another variable which is interpreted 
as a speed variable. We thus consider the stochastic differential 
equation on $\R^{n}\times \R^n$ given by 
\begin{equation}\label{eq_diff_kin}
\left\{
\begin{array}{l}
d X_t = Y_t \, dt \\
d Y_t = \sqrt{2\beta} \, dW_t - \beta Y_t \,dt - \nabla V(X_t) \, dt .
\end{array} 
\right. 
\end{equation}
where $(W_t)$ is a standard Brownian motion 
on $\R^n$ and $\beta$ is a positive parameter, 
called the \emph{friction} parameter hereafter. 
Note that this equation is degenerate, in the sense that we only 
have a diffusion on the speed variable and not on the space 
variable. The kinetic version of the Langevin 
diffusion is sometimes called \emph{underdamped} Langevin 
diffusion, whereas the usual diffusion~\label{eq_langevin} 
is called \emph{overdamped}. We will use both terminologies here, 
so underdamped and overdamped are synonyms for kinetic and non kinetic, 
respectively. The underdamped diffusion
admits a unique stationary distribution, namely the measure 
$\pi := \mu \otimes \gamma$, where $\gamma$ is the standard Gaussian measure 
on $\R^n$. Under mild assumptions, the diffusion is also ergodic, and 
$(X_t,Y_t)$ converges in distribution to $\pi$. As far as sampling is concerned what matters is that the first factor is our target measure $\mu$,
so that the position $X_t$ converges to $\mu$.
The kinetic Langevin algorithm is the algorithm obtained 
by discretizing the diffusion~\eqref{eq_diff_kin}. 
Namely we fix a time step parameter $\eta$ and 
we consider the following system of equations
\begin{equation}\label{eq_euler2}
\begin{cases} 
d X^\eta_t = Y^\eta_t dt \\
d Y^\eta_t = \sqrt{2\beta}\, d W_t - \beta Y^\eta_t \, dt - \nabla V ( X^\eta_{\lfloor t/\eta \rfloor\eta } ) \, dt ,  
\end{cases}  
\end{equation} 
where $\eta$ is the time step and $\lfloor \cdot \rfloor$ denotes the integer 
part. Thus the only difference with~\eqref{eq_diff_kin} is that in the equation 
for the speed, the gradient is not queried at the current position but 
at some past position, corresponding to 
latest integer multiple of $\eta$. We initiate this at some (possibly random) 
point $(x_0,y_0)$ and we set $(x_k,y_k) = ( X^\eta_{k\eta}, Y^\eta_{k\eta})$
for every integer $k$. The process $(x_k,y_k)$ is Markov chain whose transition kernel is explicit. Namely, the transition measure at point 
$(x,y)$ is the Gaussian measure on $\R^n \times \R^n$, 
centered at point $(x',y')$ 
and with covariance matrix 
$\begin{pmatrix} a & c \\ c  & b \end{pmatrix} \otimes I_n$ 
where
\[ 
\begin{split} 
& x' = x + \frac{1-\e^{-\beta \eta}} \beta y - \frac{\e^{-\beta \eta} - 1 + \beta \eta}{\beta^2} \nabla V (x) \\
& y' = \e^{-\beta \eta} y - \frac{ 1 - \e^{-\beta \eta}}\beta \nabla V (x) \\
& a = \frac 1 {\beta^2} ( 4 \e^{-\beta \eta} - \e^{-2\beta \eta} + 2 \beta \eta - 3 ) \\
& b = 1 - \e^{-2\beta \eta} \\
& c = \frac 1{\beta} ( 1 - \e^{-\beta \eta} )^2 .  
\end{split} 
\] 
Indeed the solution of the 
equation~\eqref{eq_euler2} on $[0,\eta)$ is completely explicit. Namely, 
if we start from $(x,y)$ then we have
\[
Y^\eta_t = \e^{-\beta t} y - \frac{1-\e^{-\beta t}}{\beta}\nabla V(x) 
+ \sqrt{2\beta} \int_0^t \e^{\beta(s-t)} \, dW_s , \quad \forall t \leq \eta . 
\] 
From this we also get an explicit formula 
for $X^{\eta}_t$.
This shows that conditioned on the initial 
point $(x,y)$ the random vector $(X^\eta_\eta, Y^\eta_{\eta})$ is 
a Gaussian vector on $\R^{n} \times \R^n$. Finding the different parameters 
is only a matter of computation which is omitted here. See e.g. \cite{cheng_et_al,ZCLBE} for more details. 

Thus the transition kernel is just a Gaussian kernel 
with a somewhat intricate but completely explicit 
covariance matrix, which depends only on the the friction 
parameter $\beta$ and the time step $\eta$, and not 
on the potential $V$. The potential $V$ only appears via 
its gradient in the center of mass of the Gaussian 
kernel. Each step of this Markov chain is thus easy to 
sample, under the assumption that we have an oracle for 
$\nabla V$. 

This Markov chain is what we call the kinetic Langevin algorithm 
associated to $\mu$. It depends on two parameters, the 
friction parameter $\beta$ and the discretization parameter $\eta$. 
We will show that if those parameters are chosen appropriately 
then after a polynomial number of steps the 
distribution of $(x_k,y_k)$ 
is very close to $\pi$.

\subsection{Main result} 
We will quantify the sampling error in terms of total variation distance: 
if $\mu$ and $\nu$ are two probability measures defined on the 
same space $E$ equipped with some $\sigma$-field $\mathcal A$ then 
\[ 
TV ( \mu ,\nu ) = \sup_{A\in \mathcal A} \{ \vert \mu (A) -\nu(A)\vert \}. 
\] 
A related notion is the chi-square divergence, defined by 
\[ 
\chi_2 ( \mu \mid \nu ) = \int_E  \frac{ d\mu}{d\nu} \, d \mu -1  ,  
\] 
assuming that $\mu$ is absolutely continuous with respect to $\nu$. 
If this is not the case the convention is that the chi-square divergence 
is $+\infty$. The chi-square divergence is not a distance, it is not symmetric 
and it does not satisfy the triangle inequality either. However it 
controls the total variation distance. Indeed we have  
\[ 
TV ( \mu , \nu ) \leq \sqrt{ \log ( 1 + \chi_2 ( \mu \mid \nu ) ) } \leq \sqrt{ \chi_2 ( \mu \mid \nu ) }. 
\]
By a slight abuse of notations, when $X,Y$ are random vectors
we shall often write $TV ( X,Y )$ for $TV(\mu,\nu)$ 
where $\mu,\nu$ are the respective laws of $X,Y$. We will 
adopt a similar convention for the chi-square divergence. 
For instance we shall write $\chi_2 ( X \mid \nu)$ 
for the relative chi-square divergence of the law of $X$
with respect to the measure $\nu$. 
 
There will be two hypotheses for the target measure. First of all 
we assume that $\mu$ satisfies the Poincar\'e inequality. Namely we assume that there is 
a constant $C_P$ such that for any locally Lipschitz function $f$ 
we have 
\begin{equation}\label{eq_poincare} 
\var_\mu (f) \leq C_P \int_{\R^n} \vert \nabla f\vert^2 \,d\mu .
\end{equation} 
The second hypothesis is a regularity hypothesis: We assume 
that the potential $V$ of $\mu$ is $\mathcal C^2$-smooth with Lipschitz
gradient. The Lipschitz 
constant of $\nabla V$ will be denoted by $L$ throughout. 
We shall give two estimates. 
One that is valid under the above two assumptions only, and one slightly 
better, under the additional assumption that $\mu$ is log-concave, 
in the sense that the potential $V$ is a convex function. 
\begin{theorem}\label{thm_main}
Suppose that $\mu$ satisfies Poincar\'e with contant $C_P$ and has $\mathcal C^2$-smooth and gradient Lipschitz potentiel, with constant $L$. 
Assume that the kinetic Langevin algorithm is initiated at $(x_0,y_0)$ 
with $x_0$ independent of $y_0$, and $y_0$ distributed according to the 
standard Gaussian measure on $\R^n$. Fix the friction 
parameter $\beta = \sqrt L$, and set the time step parameter $\eta$
and the number of steps $k$ by
\[ 
\begin{split}
\eta & = c \cdot  \epsilon L^{-1} C_P^{-1/2} \cdot 
\left( \sqrt{n} + \sqrt{ \log (1+\chi_2 (x_0\mid \mu )) } \right)^{-1} \cdot \log^{-1/2} \left(\frac{ \chi_2 (x_0\mid \mu) }{\epsilon } \right) ,   \\
k & = C  \cdot \epsilon^{-1} (LC_P)^{3/2} \cdot 
\left( \sqrt{n} + \sqrt{ \log (1+\chi_2 (x_0\mid \mu )) } \right) \cdot \log^{3/2} \left(\frac{ \chi_2 (x_0\mid \mu) }{\epsilon } \right) , 
\end{split}  
\]  
where $c,C$ are universal constants. 
Then we have 
\[ 
TV(x_k,\mu) \leq \epsilon . 
\] 
If in addition $\mu$ is log-concave then we have the same result with 
friction $\beta = C_P^{-1/2}$, time step $\eta$ as above, and 
number of steps $k$ given by
\[ 
k = C  \cdot \epsilon^{-1} LC_P \cdot 
( \sqrt{n} + \sqrt{ \log (1+\chi_2 (x_0\mid \mu )) } ) \cdot \log^{3/2} \left(\frac{ \chi_2 (x_0\mid \mu) }{\epsilon } \right) . 
\]  
\end{theorem} 
Some comments are in order. First of all the constants $c,C$ could be explicitly 
tracked down from the proof, but we have chosen not to do so in order to 
lighten the exposition. In particular the proof has not been optimized so as
minimize the value of $C$. That being said, our proof does not yield horribly 
large constants either, see the discussion right after Theorem~\ref{thm_hypo} below.

Secondly, 
the result depends on some \emph{warm start} hypothesis. Namely 
the algorithm must be initiated at some point $x_0$ for which we have 
some control on the chi-square divergence 
to the target measure. Assuming that this chi-square divergence is 
polynomial in the dimension, 
we then get 
\[ 
O^*( \epsilon^{-1} ( L C_P )^{3/2} \cdot n^{1/2} )
\] 
complexity for the kinetic Langevin algorithm in general and 
$O^*( \epsilon^{-1} L C_P \cdot n^{1/2} )$ complexity in the 
log-concave case. The notation $O^*$ means up to a multiplicative 
constant and possibly  poly-logarithmic factors. This is a common notation when studying the complexity of algorithms. Note that even in the more pessimistic 
scenario were the chi-square divergence of the initial distribution is exponentially 
large in the dimension, Theorem~\ref{thm_main} still provides polynomial bounds, though with a worst dependence on the dimension. 

Coming back to the situation where a polynomial warm-start is available, 
our result should be compared with 
the complexity one gets for the overdamped version of 
the algorithm under the same set of hypothesis. 
The state of the art is Dalalyan~\cite{dalalyan1}, 
which gives convergence after $O^*(\epsilon^{-2} (L C_P)^2 n )$ 
steps of the algorithm. The result in~\cite{dalalyan1} is 
not quite written this way but the above complexity 
does follow from the proof, see the discussion at 
the end of the introduction of~\cite{lehecAAP}. 
Well, in this discussion the assumption 
is that $\mu$ satisfies the log-Sobolev inequality, 
and that we have a warm start in the relative entropy sense, 
but the argument is exactly the same assuming Poincar\'e 
and a chi-square divergence warm start assumption.

The most important improvement given by the kinetic version
is probably the dependence in the dimension, which drops from $O(n)$ to 
$O(\sqrt n)$, but it should be noted that also the dependence on the precision $\epsilon$, on the Lipschitz constant $L$ of the potential and on the 
Poincar\'e constant $C_P$ of the measure are definitely 
improved by the kinetic version of the 
algorithm, all the more so in the log-concave case. 
 
\subsection{The hypocoercive estimate} 
The proof of the main theorem splits into two parts: firstly we need 
to estimate the speed of convergence of the true diffusion towards its 
equilibrium measure and secondly to control the discretization error. For 
the first part the difficulty lies in the fact that the diffusion 
is degenerate, in the sense that there is only a diffusion 
term in the speed variable and not in the space variable. 
This implies
that we cannot expect an exponential decay of the 
chi-square divergence along the diffusion of the form: 
\[ 
\chi_2 ( \nu P_t \mid \pi ) \leq \e^{-ct} \chi_2 ( \nu \mid \pi ) , 
\] 
where $c$ is the positive constant. Here $\nu$ is 
some probability measure on the product space $\R^n \times \R^n$ 
and $(P_t)$ denotes the semigroup associated to the kinetic Langevin equation~\eqref{eq_diff_kin}. In other words $\nu P_t$ denotes the law of 
$(X_t,Y_t)$ when $(X_0,Y_0)$ is distributed according to $\nu$. 
The terminology is that the diffusion fails to be coercive. However, 
there is no obvious obstruction to having 
this inequality with a prefactor. Such estimates are called hypocoercive 
estimates, and there are a number of them available in the 
literature. The result that we shall use is essentially 
from Cao, Lu and Wang~\cite{cao_lu_wang}, but for reasons to 
be explained later on we will reprove the result from scratch rather 
than taking it for granted. 
In terms of hypothesis, the hypocoercive estimate still requires 
Poincar\'e, but a weaker condition than gradient Lipschitz 
potential is enough. Indeed, it only requires a semi-convexity property. 
Namely, the assumption is that there 
exists a constant $\kappa \geq 0$ such that the 
potential $V$ of $\mu$ has the property that 
$V(x) + \frac\kappa2 \vert x\vert^2$ 
is a convex function of $x$. Note that if $\nabla V$ is 
Lipschitz with constant $L$ then this holds true 
with $\kappa = L$ but the converse is not true.
\begin{theorem}\label{thm_hypo}
If $\mu$ has $\mathcal C^2$-smooth and semi-convex potential, 
with constant $\kappa \geq 0$, and 
satisfies Poincar\'e with constant $C_P$, then 
for every probability measure $\nu$ on $\R^n \times \R^n$ we have
\[ 
\chi_2 ( \nu P_t \mid \pi )  \leq 2 \cdot \exp \left( - c\cdot \frac{ \beta t}{ 1 + ( \beta^2 + \kappa ) C_P } \right)  
\cdot \chi_2 ( \nu \mid \pi)  , \quad \forall t >0
\] 
where $\beta$ is the friction parameter of the kinetic Langevin 
diffusion and $c$ is a positive universal constant. 
\end{theorem} 
It should be noted that most hypocoercivity results are established with a prefactor 
that depends on the initial measure $\nu$. A nice feature of this one is that the prefactor is just a universal constant. 

Again, the focus of this work is on the theoretical aspects of the kinetic Langevin 
algorithm rather than its practical implementation, but for this result we 
do provide a version of the inequality that is completely 
free of hidden constants, namely:
\[ 
\chi_2 ( \nu P_t \mid \pi )  \leq \exp \left( - \frac{ \beta t}{ 10 \cdot \left( 3 + \beta \sqrt{C_P} + 2 \sqrt{1 + \kappa C_P} \right)^2 } + \frac 1{60} \right)  
\cdot \chi_2 ( \nu \mid \pi)  , \quad \forall t >0 . 
\]  
This is obtained by simply keeping track of the various constants 
involved in the proof of the theorem, see section~\ref{sec_hypo} below. 
\subsection{Related works}
The fact that the kinetic version of the algorithm
has better performance was already observed in a number of 
works. Maybe the first theoretical results in this 
direction are~\cite{gad_mic,gui_mon} in which the case where the 
target measure is Gaussian is studied in details. Of course a Gaussian 
target measure is of little interest for sampling but the point 
there was to show that when the friction parameter is set appropriately then 
the kinetic diffusion is faster than the overdamped one. 
As far as sampling is concerned, relevant references 
include~\cite{dal_riou,cheng_et_al} in which 
non asymptotic bounds for the kinetic Langevin algorithm 
are established. The two results are very similar and differ in 
two ways from the present article. Firstly the performance is measured 
in terms of the Wasserstein $2$ distance (i.e. the optimal transportation 
associated to the cost function $\vert x-y\vert^2$) 
rather than total variation. Also the hypothesis is more stringent, the 
target measure is assumed to be uniformly log-concave and 
gradient Lipschitz. In other words the potential $V$ is assumed to satisfy 
\begin{equation}\label{eq_9191}
m I_n \leq \nabla^2 V (x) \leq M I_n 
\end{equation}
for all $x$ and where $m,M$ are positive constants. Notice that 
such a potential is gradient Lipschitz with constant $L\leq M$ 
and satisfies Poincar\'e with constant $C_P \leq m^{-1}$. 
Nevertheless~\cite{dal_riou,cheng_et_al} establish that 
all things equal (namely under the assumption~\eqref{eq_9191}
and as far as the Wassertein distance is concerned)
the underdamped version outperforms the best bounds available for 
the overdamped algorithm, which were established previously 
in~\cite{durm_moul,dalalyan2}. 
In particular it is shown that the dimension dependence drops from $O(n)$ to 
$O (\sqrt n)$, as in the present work. 
In~\cite{MCCFBJ} a relative entropy estimate for the kinetic Langevin 
algorithm is established. Although relative entropy controls 
total variation (see section~\ref{sec_discr} below) this
does not recover our main result. First of all the result from~\cite{MCCFBJ} 
is proven under the assumption that the target measure satisfies a 
log-Sobolev inequality, which is a stronger hypothesis than Poincaré, 
and more importantly it is only partly 
quantitative, in the sense that the dependence on certain parameters 
of the problem is not made explicit. The reference that comes 
closer to our work is \cite{ZCLBE} (which we were not aware of 
until the first version of this paper was released). 
At a high level the results and methods of 
proof there are very similar to what is done in the current work. 
However our analysis of the discretization error is much 
simpler and also allows to capture more accurately the dependence 
on the initial condition. In~\cite{ZCLBE} the convergence is 
established for a specific warm start condition for which both 
the log chi-square divergence and the Fisher information are order $n$ 
(essentially). In such a situation the authors of~\cite{ZCLBE} obtain convergence in total variation after $O^* ( \epsilon^{-1} (L C_P)^{3/2} n^2 )$ steps of the algorithm in general and $O^* ( \epsilon^{-1} L C_P n^2 )$ steps in the log-concave case. While it is true that when $\log \chi_2 (x_0\mid\mu) = O^* (n)$ we get exactly the same complexity, our result has the advantage of not requiring bounded Fisher information. Also we get a better dimension dependence if one happens to have a better warm start
hypothesis. In particular as we mentioned already the dimension 
dependence becomes as low as $O^* (\sqrt n)$ when the initial chi-square divergence is polynomial in the dimension. This does not seem to follow from 
the analysis of~\cite{ZCLBE}.   

Let us discuss also the literature on the Hamiltonian Monte Carlo algorithm, which is a sampling  algorithm very much related to the kinetic Langevin algorithm. 
It is based on the observation that the Hamiltonian dynamic 
\begin{equation}\label{eq_HMC}
\begin{cases} 
y'(t) = x(t) \\
x'(t) = - \nabla V (y(t)) 
\end{cases} 
\end{equation}
preserves the Lebesgue measure as well as the potential 
$\mathcal H(x,y) := V(x) + \frac{\vert y\vert^2}2$. As a result it also 
preserves the probability measure $\pi$. The HMC process is the 
piecewise deterministic process obtained by choosing a time step 
$\delta$, resampling the speed $y$ at every integer multiple 
of $\delta$ and following the system of equations~\eqref{eq_HMC} 
in between. Of course this does not admit an explicit 
solution and in order to turn this ideal HMC dynamic into a proper 
algorithm one needs to replace the Hamiltonian dynamic phase 
by a Euler type discretization. The resulting algorithm looks a lot like the kinetic Langevin algorithm, and in particular the two algorithms should have pretty much the same performance.
Non asymptotic theoretical bounds for the HMC algorithm 
are established in~\cite{MS1,MS2,chen_vempala}. As 
for the works on the kinetic Langevin algorithm mentioned above, 
they prove convergence estimates in the Wasserstein sense and 
under the assumption that the Hessian of the potential of the target measure
is bounded from above and below (by a positive constant). The result from~\cite{chen_vempala} 
gives essentially the same estimate as what~\cite{cheng_et_al,dal_riou} 
get for the kinetic Langevin algorithm.  

We have only presented a very short selection of 
the literature on the Langevin and the HMC algorithms. 
For instance there are also many references where the focus 
is on the discretization scheme. 
Indeed, we used the most natural one in this paper but there 
exist variants in the literature. These typically yield better 
dependence on the error $\epsilon$ at the cost of 
higher order regularity estimates on the potential. 
See~\cite{CDMS} and the references therein.    

\subsection{Perspectives} 
Firstly we conjecture that Theorem~\ref{thm_main} should hold true for 
the HMC algorithm as well. 
Maybe more interestingly, one drawback  with our result is that we lose regularity from the hypothesis to the conclusion. We prove a total 
variation estimate under a warm start hypothesis in the chi-square sense. 
It would be more satisfactory to get a chi-square estimate in the conclusion as well. Note that the chi-square divergence controls the Wasserstein distance under 
Poincaré (see~\cite{Liu}). Therefore such a result would imply also a convergence 
estimate for the Wasserstein distance, similar to that from the aforementioned works~\cite{dal_riou,cheng_et_al}, but under significantly weaker hypothesis on the target measure. 
In the same way, it would be interesting to have an analogue result for the relative entropy (both in the hypothesis and 
in the conclusion) under a log-Sobolev hypothesis for the target measure. 
In the overdamped version of the algorithm this task was
completed by Vempala and Wibisono~\cite{VW} but it 
is not clear at all whether their approach can be adapted to the kinetic 
Langevin algorithm. 

\paragraph{Acknowledgment.} The author is grateful to 
Arnaud Guillin, Pierre Monmarch\'e and Matthew Zhang
for discussions related to this work and/or 
pointing out relevant references. We would also like to thank 
the two anonymous referees for their careful reading of the manuscript
and accurate comments. 

\section{The hypocoercive estimate} \label{sec_hypo}
It\^o's formula shows that the generator of the kinetic 
diffusion~\eqref{eq_diff_kin} is the operator 
\[ 
\mathcal L f(x,y) = \beta \Delta_y f (x,y) + \nabla_x f (x,y) \cdot y 
-  \beta \nabla_y f (x,y) \cdot y - \nabla_y f (x,y) \cdot \nabla V(x) .    
\]  
Integrating by parts we see that the probability measure $\pi$ 
on $\R^n \times \R^n$ whose density is proportional to 
\[ 
\e^{-V(x)} \cdot \e^{ - \frac{\vert y\vert^2}{2} } 
\] 
is stationary. One important fact is that the diffusion $(X_t,Y_t)$ is 
not reversible. In terms of the operator $\mathcal L$, this means 
that the operator is not symmetric in $L^2 (\pi)$. More
precisely, direct calculations show that the part  
\[ 
\Delta_y f (x,y) -  \nabla_y f (x,y) \cdot y =: \mathcal L_{\rm OU} f
\] 
is symmetric and that the part 
\[  
\nabla_x f (x,y) \cdot y 
 - \nabla_y f (x,y) \cdot \nabla V(x) =: \mathcal L_{\rm Ham} f
\] 
is antisymmetric. The indices \emph{OU} and \emph{Ham} stand 
for Ornstein-Ulhenbeck and Hamiltonian, respectively. 
As a result the adjoint operator is $\mathcal L^* = \beta \mathcal L_{\rm OU} - \mathcal L_{\rm Ham}$. 
We let $(P_t)$ be the semigroup with generator $\mathcal L$, 
and $(P_t^*)$ be the adjoint semigroup. In other words, 
if $\nu$ is a probability measure on $\R^n \times \R^n$ which 
is absolutely continuous with respect to $\pi$, 
with density $f$, then $\nu P_t$ has density $P_t^* f$ 
with respect to $\pi$. Moreover the chi-square divergence 
between $\nu$ and $\pi$ is nothing but the variance of the relative 
density, so that $\chi_2( \nu P_t \mid \pi) = \var_\pi ( P_t^* f )$. 
Therefore Theorem~\ref{thm_hypo} can be reformulated as follows. 
\begin{theorem} \label{thm_hypo2}
If the potential of $\mu$ is $\mathcal C^2$-smooth and semi-convex, with constant $\kappa\geq 0$, and if $\mu$ satisfies
Poincar\'e with constant $C_P$, then for every function 
$f\in L^2 (\pi)$ we have
\[ 
\var_\pi ( P_t^* f ) \leq 2 \exp \left( -c \cdot \frac{ \beta t} { 1 +(\beta^2+\kappa) C_P } \right)  \var_\pi (f) .  
\]
\end{theorem} 
As we mentioned in the introduction the difficulty arises 
from the degeneracy of the diffusion. 
Let us say a few more words about this. By definition 
\[ 
\partial_t P_t^* f = \mathcal L^* P_t^* f . 
\] 
Integrating by parts we see that
\[ 
\frac{d}{dt} \var_\pi (P_t^* f) 
= 2 \int_{\R^{2n}} (\mathcal L^* P_t^* f) P_t^* f \, d\pi 
= - 2\beta  \int_{\R^{2n}} \vert \nabla_y P_t^* f \vert^2 \, d\pi.  
\] 
This shows that the variance of $P_t^* f$ is 
non-increasing in time. 
If we want a more quantitative statement, 
the issue is that only the gradient in $y$ appears in the 
the dissipation of the variance, and not the full gradient.
This comes from the fact that the associated diffusion is degenerate and
that the Brownian term only appears in the $y$ variable. 
Thus we cannot hope to lower bound the dissipation of variance 
by the variance itself. There are a number of 
ways around this issue, and this line of research usually 
goes by the name of \emph{hypocoercivity}. This was 
pioneered by Villani~\cite{villani}, other classical references 
include~\cite{DMS,baudoin} to name only a very few of them. 
In Villani's work the main idea is 
to consider the dissipation of some perturbed energy, of the form 
\[ 
\mathcal E (f_t ) := \var_\pi ( f_t ) 
+ \int_{\R^n \times\R^n} \langle A \nabla f_t , \nabla f_t \rangle \, d \pi 
\]
where $A$ is a suitably chosen positive semi-definite matrix. This is 
also the approach taken by many of the subsequent works, including 
the works having application to sampling~\cite{dal_riou,cheng_et_al} 
which we already mentioned. 
To arrive at Theorem~\ref{thm_hypo2} 
we take a slightly different route, which is
inspired by the work of Albritton, Armstrong, Mourrat and Novack~\cite{AAMN}. The main idea is that while 
an inequality of the form $\var_\pi (f) \leq C \int \vert \nabla_y f\vert^2 \, d\pi$ is impossible, if we integrate this on some 
time interval along the kinetic Langevin diffusion, then the inequality 
becomes plausible. This is called \emph{space-time} Poincar\'e inequality. 
This approach is quite general and not restricted 
to the case of the kinetic Langevin diffusion. However the results from~\cite{AAMN}  
are mostly qualitative, whereas here we need quantitative estimates with explicit dependence on all parameters of the problem. A quantitative version of~\cite{AAMN} was developed by Cao, Lu and Wang in~\cite{cao_lu_wang} and the 
proof spelled out below is very much inspired by their argument. 

Before embarking for the proof of Theorem~\ref{thm_hypo2} let us 
remark that it is invariant by scaling. Indeed, notice that if $(X_t,Y_t)$ 
is a kinetic Langevin diffusion associated to $\mu$ 
and with friction parameter $\beta$, then 
$(\frac1\lambda X_{\lambda t} , Y_{\lambda t})$ 
is a kinetic Langevin diffusion associated to
the measure  $\mu$ scaled by $1/\lambda$ (i.e. the law of $X/\lambda$ if
$X$ is a vector with law $\mu$) and friction parameter $\lambda \beta$. This implies easily that if we define $\alpha = \alpha (t,\beta, \mu)$
to be the best estimate one could get for $\mu$ in this theorem, 
namely 
\[ 
\alpha (t,\beta,\mu) = \sup \left\{ \frac{ \var_\pi (P^*_tf) }{\var_\pi (f) } \right\} , 
\] 
where the supremum is taken over every function $f\in L^2(\pi)$, 
then $\alpha$ has the property that 
\[ 
\alpha ( t , \beta, \mu ) = \alpha ( t / \lambda , \beta \lambda, \mu_{1/\lambda} ) , 
\]
where $\mu_{1/\lambda}$ denotes the measure 
$\mu$ scaled by $1/\lambda$. The upper bound for $\alpha$ 
provided by Theorem~\ref{thm_hypo2} is also invariant by this
change of variable. As a result it is enough to prove the result  
when one of the parameters has a prescribed value, we will exploit 
this observation later on.    

\subsection{The $L^2$ method for Poincar\'e} 
In this section we gather some well known facts about the Laplace 
operator associated to some measure that will be needed later on. 
Let $\mu$ be a probability measure on $\R^n$, of 
the form 
\[ 
\mu (dx)= \e^{ - V(x) } \, dx 
\] 
for some $\mathcal C^2$-smooth function $V\colon \R^n \to \R$. 
The Laplace operator associated to $\mu$ 
is the differential operator $L_\mu$ defined by 
\[ 
L_\mu f = \Delta f - \langle \nabla f , \nabla V \rangle  . 
\] 
Originally $L_\mu$ is defined on the space of $\mathcal C^\infty$-smooth 
and compactly supported functions $f$, in which case 
an integration by parts gives 
\[ 
\int_{\R^n} (L_\mu f) g  \, d\mu = 
- \int_{\R^n} \langle \nabla f ,\nabla g \rangle \,d\mu  .  
\]
This shows in particular $\langle L_\mu f , g \rangle_{L^2(\mu)} 
= \langle f, L_\mu g\rangle_{L^2(\mu)}$ and that 
$\langle L_\mu f , f \rangle_{L^2(\mu)}\leq 0$ 
for any functions $f,g$ in the domain of $L_\mu$. 
In the language of operator theory, the operator $-L_\mu$ 
is said to be symmetric and monotone. 
It turns out that $L_\mu$ admits a unique self-adjoint extension, 
the domain of which contains $H^1(\mu)$, and that the 
above integration by parts is true for every $f,g$ in $H^1(\mu)$.
Recall that $H^1(\mu)$ is the space of functions $f\in L^2 (\mu)$ 
whose weak gradient also belongs to $L^2 (\mu)$. The operator $L_\mu$ 
admits $0$ as a simple a eigenvalue, the corresponding eigenspace 
consists of constant functions. We say that $L_\mu$, or 
rather $-L_\mu$, has a spectral gap if there exists $\lambda_0>0$ 
such that the rest of the spectrum of $-L_\mu$ is included in some 
interval $[\lambda_0;+\infty)$. This property turns out to be equivalent 
to the Poincar\'e inequality, as we shall see now. 
\begin{lemma} \label{lemma_bochner1}
The measure $\mu$ satisfies Poincar\'e with constant $C_P$ 
if and only if the spectral gap of the operator $L_\mu$
is at least $C_P^{-1}$. Moreover this is also equivalent to 
the inequality
\[ 
\int_{\R^n} \vert \nabla f \vert^2 \, d\mu 
\leq C_P \cdot \int_{\R^n} (L_\mu f)^2 \, d\mu ,  
\] 
for every $f$ in the domain of $L_\mu$. 
\end{lemma} 
\begin{lemma}[Bochner formula]\label{lemma_boc2}
For every smooth and compactly supported 
function $f$ we have 
\[ 
\int_{\R^n} (L_\mu f)^2 \,d \mu 
= \int_{\R^n} \Vert \nabla^2 f \Vert^2_{F} + \nabla^2 V ( \nabla f, \nabla f) \, d\mu.  
\]
\end{lemma} 
Here and in the sequel 
\[ 
\Vert A\Vert_{F} = \left( \sum_{i,j} a_{ij}^2 \right)^{1/2} 
\] 
denotes the Frobenius norm of a symmetric matrix $A= (a_{ij})$. 
If we combine together these two lemmas
we easily get the so-called Lichnerowicz estimates: 
If the potential $V$ of $\mu$ is uniformly convex, in the sense that
there exists a constant $\alpha >0$ such that $\nabla^2 V \geq \alpha \cdot \id$ 
pointwise and for the order given by the cone of positive 
semi-definite matrices, then the measure $\mu$ satisfies 
Poincar\'e with constant $1/\alpha$. 
Indeed Bochner's formula then implies 
\[ 
\int_{\R^n} (L_\mu f)^2 \,d \mu 
\geq \alpha \int_{\R^n} \vert \nabla f \vert \, d\mu 
\] 
which yields the claimed bound on the Poincar\'e constant thanks to 
Lemma~\ref{lemma_bochner1}. 

This approach for the Poincar\'e inequality 
dates back to the work of Lichnerowicz~\cite{lichnerowicz}. 
More recent references where this method plays a role include~\cite{CEFM,klartag_berry} among others. 
We refer to these for the proofs 
of Lemmas~\ref{lemma_bochner1}
and~\ref{lemma_boc2}. 
For our purposes the following immediate 
consequence of the two lemmas will be important. 
\begin{corollary} \label{cor_bochner} 
If $\mu$ satisfies Poincar\'e with constant $C_P$ 
and if the potential $V$ of $\mu$ is $\mathcal C^2$-smooth 
and semi-convex with constant $\kappa$, 
then for every $f$ in the domain of $L_\mu$ we have 
\[ 
\int_{\R^n} \Vert \nabla^2 f \Vert^2_{F} \, d\mu \leq 
\left( 1 + \kappa \cdot C_P \right) \int_{\R^n} ( L_\mu f)^2 \, d\mu . 
\]
\end{corollary} 
\begin{proof} 
Bochner's formula and the semi-convexity hypothesis yield 
\[ 
\int_{\R^n} (L_\mu f)^2 \, d\mu 
\geq \int_{\R^n} \Vert \nabla^2 f\Vert_F^2 \, d\mu - \kappa  
\int_{\R^n} \vert \nabla f\vert^2 \, d\mu . 
\]
By Lemma~\ref{lemma_bochner1}
\[ 
\int_{\R^n} \vert \nabla f\vert^2 \, d\mu \leq C_P \int_{\R^n} (L_\mu f)^2 \, d\mu ,
\] 
and the result follows. 
\end{proof}

\subsection{A divergence equation} 
The next proposition, taken 
from the Cao, Lu, Wang paper, 
is the main ingredient in the proof of the 
hypocoercive estimate. This proposition 
is revisited and extended to a somewhat more general context 
in the preprint~\cite{guillin}. However, instead 
of just taking the result for granted 
we shall reprove it. One reason is that in the aforementioned 
works the proposition is proven under annoying additional 
technical assumptions which are actually not needed. 
Maybe more importantly the existing proofs are mostly computational 
and we find them hard to follow. Our proof relies on spectral theory and is arguably more conceptual. This hopefully sheds new light on this somewhat delicate 
proposition. 
\begin{proposition}\label{prop_key} 
Suppose that the potential of $\mu$ is $\mathcal C^2$-smooth and semi-convex, with constant $\kappa$, 
and that $\mu$ satisfies Poincar\'e with constant $1$. 
Fix a time horizon $T$ and let 
\[ f \in L^2 ( [0,T] \times \R^n , \lambda \otimes \mu) \]
where $\lambda$ is the Lebesgue measure, and 
assume that $f$ is orthogonal to constants. 
Then there exist $u$ and $v$ in $L^2 ( \lambda \otimes \mu)$ 
satisfying the Dirichlet boundary conditions, namely $u(0,\cdot) = u(T,\cdot)= 0$, and similarly for $v$, such that
\begin{equation}
\label{eq_laplace}
\partial_t u + L_\mu v = f  , 
\end{equation} 
and such that the  
following estimates hold true: 
\begin{enumerate}[i)]
\item $\Vert \nabla u \Vert_{L^2(\lambda\otimes \mu)} \lesssim T \Vert f \Vert_{L^2(\lambda\otimes \mu)}$,
\item $\Vert \nabla \partial_t v \Vert_{L^2(\lambda\otimes \mu)} \lesssim T^{-1} \Vert f\Vert_{L^2(\lambda\otimes \mu)}$,
\item $\Vert \nabla v \Vert_{L^2(\lambda\otimes \mu)} \lesssim \Vert f\Vert_{L^2(\lambda\otimes \mu)}$,
\item $\Vert \nabla^2 v \Vert_{L^2(\lambda\otimes \mu)} \lesssim (1+\kappa)^{1/2} \Vert f\Vert_{L^2(\lambda\otimes \mu)}$.   
\end{enumerate} 
\end{proposition} 
Here, the notation $\lesssim$ means up to a multiplicative 
universal constant. Also, when applied to tensors, 
norms should be interpreted coordinate-wise. 
For instance 
\[ 
\Vert \nabla^2 v \Vert_{L^2(\lambda\otimes \mu)}^2 = \sum_{ij} \Vert \partial_{ij} v\Vert_{L^2(\lambda\otimes \mu)}^2 . 
\]

The results from the previous subsection section will play a role at some 
point but there is still some work to be done here.  
The main difficulty is that there is a mismatch between the 
boundary condition for $u$ and that for $v$.

As a preliminary step to the proof of Propostion~\ref{prop_key}, 
we notice that it is enough to consider the case where $T=1$.
Indeed given three functions $u,v,f$ in  
\[ 
L^2 ( [0,T]\times \R^n , \lambda \otimes \mu) , 
\] 
we let $\tilde u,\tilde v, \tilde f$ be the functions 
in $L^2 ( [0,1]\times \R^n , \lambda \otimes \mu)$ given by 
\[ 
\begin{split} 
& \tilde f (t,x) = f ( tT,x )  \\
& \tilde u(t,x) = T^{-1} \cdot  u ( tT , x )  \\
& \tilde v(t,x) = v ( tT , x ) . 
\end{split} 
\] 
Obviously $f$ is orthogonal to constants if and only $\tilde f$ is, and $u,v$ 
satisfy the Dirichlet boundary conditions if and only if $\tilde u,\tilde v$ do. 
It is also clear that the divergence
equation~\eqref{eq_laplace} for $(u,v,f)$ is equivalent to that for 
$(\tilde u,\tilde v,\tilde f)$. Lastly, elementary computations show that the Sobolev type estimates $i),ii),iii),iv)$ for $u,v,f,T$ amount to the same estimates for 
$\tilde u,\tilde v,\tilde f,1$. Therefore we assume that $T= 1$ from now on. 
\begin{remark}
In the same way, while we stated this proposition under the assumption 
that $C_P = 1$ (which will be sufficient for our needs later on), 
one can show a more general estimate 
that depends on the Poincar\'e constant of the measure $\mu$
by rescaling the space variable.
\end{remark} 
%
It will be convenient to introduce the following notations. 
We let $A = -\partial_t \partial_t$ with Dirichlet boundary conditions. 
By a slight abuse of notation we will view it either as an 
operator on $L^2 ( \lambda)$ or as an operator on 
$L^2 ( \lambda \otimes \mu )$ that acts on the time variable only. 
In both cases it is a 
self adjoint positive semi-definite unbounded operator. 
We also let $B = - L_\mu$, and again we view this either as 
a positive semi-definite operator on $L^2(\mu)$ or on 
$L^2 (\lambda\otimes \mu)$. 
%
%
%
We also denote by $L^2_0$ the subspace of $L^2$ consisting of
functions which are orthogonal to constants. 
Lastly for $h \in L^2 (\lambda\otimes \mu)$ we define 
\[ 
\Pi h(x) = \int_0^1 h(t,x) \, dt . 
\] 
In other words $\Pi h$ is the second marginal
of $h$. Note that $\Pi \colon L^2 (\lambda \otimes \mu) \to L^2 (\mu)$ 
is a bounded operator, and that its adjoint $\Pi^*$ is the canonical 
injection of $L^2 ( \mu )$ into $L^2 (\lambda \otimes \mu)$. Observe 
that both $\Pi$ and $\Pi^*$ preserve $L^2_0$. 
The next lemma is the key to the proof of Proposition~\ref{prop_key}. 
\begin{lemma} \label{lem_step789}
Recall that we assume that $T=1$. Consider the following 
operators 
\begin{itemize}
\item $Q_1 := \Pi AB(A^2+B^2)^{-1} \Pi^*$;
\item $Q_2 := \Pi A^4 (A^2+B^2)^{-2} \Pi^*$;
\item $Q_3 := \Pi A^2 (A^2+B^2)^{-1}$.   
\end{itemize}    
The following properties hold true:
\begin{enumerate}[a)]
\item Both $Q_1$ and $Q_2$ are bounded and positive semi-definite on 
$L^2(\mu)$ and leave the subspace $L^2_0 (\mu)$ invariant. The operator 
$Q_3$ is bounded from $L^2 (\lambda\otimes \mu)$ to $L^2 (\mu)$ 
and maps $L^2_0 (\lambda \otimes \mu)$ to $L^2_0(\mu)$.  
\item As operators on $L^2_0(\mu)$ the operators $Q_1,Q_2$ satisfy $Q_2 \leq C \cdot Q_1$ for the order given by the cone of positive semi-definite operators, 
and where $C$ is a universal constant.  
\item As operators on $L^2_0 (\mu)$ both $Q_1$ and $Q_2$ are actually 
positive definite. 
\item On $L^2_0(\mu)$, we also have $0 \leq Q_3^* Q_1^{-1} Q_3 \leq C \cdot\id$.
\end{enumerate} 
\end{lemma} 
\begin{proof}
We start with claim $a)$. 
Since $A$ and $B$ commute and are positive semi-definite 
we have 
\[ 
0 \leq AB (A^2 + B^2 )^{-1} \leq \frac 12 \id .  
\] 
In particular $AB (B^2 + A^2 )^{-1}$ 
is a bounded positive semi-definite operator on $L^2(\lambda \otimes \mu)$.
Observe that $\mathbbm 1$ is an eigenfunction: 
\[ 
AB(A^2+B^2)^{-1} \mathbbm 1 = A (B^2+A^2)^{-1} B \mathbbm 1 = 0 .   
\]
Since $AB(A^2+B^2)^{-1}$ is self-adjoint, it also leaves the space $\mathrm{span}\{\mathbbm 1\}^\perp = L^2_0( \lambda \otimes \mu)$ invariant. Composing 
by $\Pi$ and $\Pi^*$ we thus see that $Q$ is a positive semi-definite bounded operator on $L^2 (\mu)$ leaving $L^2_0 (\mu)$ invariant.
In a similar way we have
\[ 
A^2  (A^2 + B^2)^{-1} = \mathbbm 1 - B^2(A^2+B^2)^{-1} \mathbbm 1 = \mathbbm 1.  
\] 
Thus $A^2 (B^2 + A^2)^{-1}$ leaves $L^2_0 ( \lambda \otimes \mu)$ invariant. 
It is also clearly positive semi-definite and bounded. This 
implies easily that 
$Q_2$ is also a positive semi-definite operator on $L^2 (\mu)$ that
preserves $L^2_0 (\mu)$, and that $Q_3 \colon L^2 (\lambda\otimes \mu) \to L^2(\mu)$ is bounded and maps $L^2_0 (\lambda \otimes \mu)$ to $L^2_0(\mu)$. 
\\
We move one to the proof of $b)$.  
We introduce the sine basis of $L^2 ([0,1],\lambda)$, namely
\begin{equation}\label{eq_sine}
e_n (t) = \sqrt 2 \cdot \sin ( \pi nt ) \, \quad n \geq 1 .
\end{equation} 
This is an orthonormal basis of $L^2([0,1],\lambda)$ for which $A$ is 
diagonal. More precisely
\[ 
A e_n = n^2 e_n , \quad \forall n \geq 1 . 
\] 
An element $f\in L^2 (\lambda \otimes \mu)$ can be written uniquely 
\begin{equation}\label{eq_mlmlml}
f = \sum_{n\geq 1} e_n \otimes f_n 
\end{equation}
where $(f_n)$ is a sequence of elements of $L^2 (\mu)$ such that 
the series $\sum_{n\geq 1} \Vert f_n\Vert_{L^2 (\mu)}^2$ is converging.
Here for $h\in L^2(\lambda)$ and $k\in L^2(\mu)$, we denote by $h\otimes k$ 
the function that maps $(t,x)$ to $h(t) k(x)$. 
Then the action of $A$ is simply given by the equation 
\begin{equation}\label{eq_sinetruc}
A f = \sum_{n\geq 1} n^2 e_n \otimes f_n .  
\end{equation}
A bit more precisely, $f$ belongs to the domain of $A$ 
if and only if the series $\sum_{n\geq 1}n^4 \Vert f_n\Vert_{L^2 (\mu)}^2$
is converging, in which case~\eqref{eq_sinetruc} holds true. 
In the same way $f$ belongs to the domain of $B$ if every function $f_n$ does
and if the series $\sum_{n\geq 1} \Vert B f_n\Vert^2_{L^2 (\mu)}$ is 
converging. When this is the case we have 
\[ 
B f = \sum_{n\geq 1} e_n \otimes ( B f_n )  . 
\] 
Therefore 
\[ 
(A^2+B^2) f = \sum_{n\geq 1}  e_n \otimes \left( (n^4 \id+B^2) f_n \right) , 
\] 
and also 
\begin{equation}\label{eq_ppppp}
AB(A^2+B^2)^{-1} f = \sum_{n\geq 1} e_n \otimes \left(  n^2 B (n^4 \id+B^2)^{-1} f_n \right) . 
\end{equation} 
The constant function admits the following decomposition in the sine basis: 
\[ 
\mathbbm 1 = \frac{2 \sqrt 2}{\pi} \sum_{n \text{ odd}} \frac{e_n}{n}. 
\]
Therefore, for $g\in L^2_0 (\mu)$ we have 
\[ 
\Pi^* g = \mathbbm 1 \otimes g = \frac{ 2 \sqrt 2}{\pi} 
 \sum_{n \text{ odd}} \frac{e_n \otimes g}n ,  
\] 
and for $f = \sum_{n\geq 1} e_n \otimes f_n \in L^2_0 (\lambda \otimes \mu)$
we have 
\[ 
\Pi f = \frac{ 2 \sqrt 2}{\pi} \sum_{n\text{ odd}} \frac{f_n}n . 
\]  
Combining with~\eqref{eq_ppppp} we thus get
\[
Q_1 = \Pi  AB (A^2+B^2)^{-1} \Pi^* = \frac{8}{\pi^2} 
\sum_{n \text{ odd}}  B( n^4 \id + B^2 )^{-1}  .   
\]  
In the same way we have  
\begin{equation}\label{eq_QQQ} 
Q_2 = \Pi A^4 (A^2 + B^2)^{-2} \Pi^* = \frac 8{\pi^2} \cdot 
\sum_{n \text{ odd}}  n^6 (n^4 \id+B^2)^{-2} . 
\end{equation}
Therefore the desired inequality $b)$ can be reformulated as 
\begin{equation}\label{eq_refo} 
\sum_{n \text{ odd}} n^6 (n^4 \id+B^2)^{-2}  
\leq  C \sum_{n \text{ odd}} B  (n^4 \id+B^2)^{-1}  ,  
\end{equation} 
as self-adjoint operators on $L^2_0(\mu)$. 
Recall that $\mu$ satisfies Poincar\'e with constant $1$, which by Lemma~\ref{lemma_bochner1} shows that the spectral gap of $B=-L_\mu$ 
is at least $1$. This means that when we restrict to $L^2_0(\mu)$ the 
spectrum of $B$ is included in the interval $[1,+\infty)$. 
We claim that this information alone yields~\eqref{eq_refo}. 
In other words the inequality would be true for any Hilbert space, and any unbounded positive definite operator $B$ whose spectrum lies in the 
interval $[1,\infty)$. 
Indeed, by the spectral theorem 
there is some resolution $(E_\lambda)$ 
of the identity of $L^2_0 (\mu)$ such that 
\begin{equation}\label{eq_spect}
B = \int_\R \lambda \, d E_\lambda . 
\end{equation} 
Moreover, since the spectrum of $B$ is above $1$, 
the integral in~\eqref{eq_spect} 
is actually supported on $[1,\infty)$ rather than on the whole line. 
Then for any integer $n\neq 0$ and any function $g\in L^2_0 (\mu)$ 
we have  
\begin{equation}\label{eq_hhhhhhdddd}
\langle B(n^4 \id + B^2)^{-1} g , g\rangle_{L^2(\mu)}
= \int_1^\infty \frac{\lambda}{n^4+\lambda^2} \, \nu_g(d\lambda) 
\end{equation} 
where $\nu_g$ is the spectral measure associated to $g$, namely the measure on 
$[1,\infty)$ whose distribution is given by 
\[ 
\nu_g ( [1,\lambda] ) = \langle E_\lambda g ,g \rangle_{L^2(\mu)}  , \quad 
\forall \lambda \geq 1 . 
\]  
There is an analogous formula for $n^6(n^4+B^2)^{-2}$ and~\eqref{eq_refo}
can thus be reformulated as 
\[ 
\int_1^\infty \sum_{n \text{ odd}} \frac{n^6}{ (n^4 +\lambda^2)^2 } \, \nu_g ( d\lambda) \leq C \int_1^\infty \sum_{n \text{ odd}}  \frac{\lambda}{ n^4 +\lambda^2} 
\, \nu_g(d\lambda) , \quad \forall g\in L^2_0 (\mu) .
\] 
But this would obviously follow from the pointwise inequality 
\begin{equation}\label{eq_pointwise}
\sum_{n \text{ odd}} \frac{n^6}{ (n^4 +\lambda^2)^2 } 
\leq C \sum_{n \text{ odd}}  \frac{\lambda}{ n^4 +\lambda^2 } , 
\quad \forall \lambda \geq 1.  
\end{equation} 
Thus all we have to do is to prove this relatively elementary 
inequality, which can be done as follows.
Since $n^6 \leq (n^4 + \lambda^2)^{3/2}$ it is 
enough to prove that
\[ 
\sum_{n \text{ odd}} \frac 1 { (n^4 +\lambda^2)^{1/2} } 
\leq C \lambda \sum_{n \text{ odd}}  \frac 1 { n^4 +\lambda^2 } , 
\quad \forall \lambda \geq 1.  
\]
Now observe that 
\begin{equation}\label{eq_smett}
\begin{split}  
\sum_{n \text{ odd}} \frac 1 { (n^4 +\lambda^2)^{1/2} } 
& \leq \sum_{n\geq 1} \frac 1 { (n^4 +\lambda^2)^{1/2} } \\
& \leq \int_0^\infty  \frac 1 { (x^4 +\lambda^2)^{1/2} } \, dx 
= C_1 \lambda^{-1/2} 
\end{split} 
\end{equation} 
where $C_1 = \int_0^\infty (x^4+1)^{-1/2} \, dx$. 
In a similar way 
\[ 
\begin{split} 
\sum_{n \text{ odd}} \frac 1 { n^4 +\lambda^2 }
& \geq \frac 12 \sum_{n\geq 1}  \frac 1 { n^4 +\lambda^2 } \\
& =  - \frac1{2\lambda^2} + 
\frac 12 \sum_{n\geq 0}  \frac 1 { n^4 +\lambda^2 } \\ 
& \geq -\frac1{2\lambda^2} + \frac12 \int_0^\infty \frac 1{x^4 + \lambda^2} \, dx 
= -\frac1{2\lambda^2} + \frac{C_2}{\lambda^{3/2} }
\end{split} 
\] 
where $C_2 = (1/2) \int_0^\infty (x^4+1)^{-1} \, dx$. This clearly 
implies that 
\begin{equation}\label{eq_mlkjhhggged}
\sum_{n \text{ odd}} \frac 1 { n^4 +\lambda^2 } \geq C_3 \lambda^{-3/2} 
\end{equation}
for all $\lambda \geq 1$ and some universal constant $C_3$.
Putting~(\ref{eq_smett}) and~(\ref{eq_mlkjhhggged}) together yields~(\ref{eq_pointwise}) and finishes the proof of~$b)$. 
\\
To prove $c)$, observe that 
if $g\in L^2_0(\mu)$ is different from $0$ then
the spectral measure $\nu_g$ is not identically zero. 
As a result  
\[ 
\langle n^6(n^4+B^2)^{-4} g , g \rangle_{L^2(\mu)} 
= \int_1^\infty \frac{n^6}{(n^4+\lambda^2)^2} \nu_g(d\lambda) >0.   
\]
Summing over odd integers and combining with~(\ref{eq_QQQ}) we 
see that $Q_2$ restricted to $L^2_0(\mu)$ is positive definite.
By claim $b)$ this implies that also $Q_1$ is positive definite
on $L^2_0(\mu)$. 
\\
Lastly, $a)$ and $c)$ imply that the operator $Q_3^* Q_1^{-1} Q_3$ is 
well-defined and positive semi-definite on $L^2_0(\mu)$. A priori this operator 
could be unbounded. However, since $\Pi^* \Pi$ is the identity map, 
we have $Q_3 Q_3^* = Q_2$. As a result 
\[ 
(Q_3^* Q_1^{-1} Q_3)^2 = Q_3^* Q_1^{-1} Q_2 Q_1^{-1} Q_3. 
\]
Now $b)$ implies that 
\[ 
Q_1^{-1} Q_2 Q_1^{-1} \leq C \cdot   Q_1^{-1} Q_1 Q_1^{-1} = C \cdot Q_1^{-1} . 
\]
Therefore $(Q_3^* Q_1^{-1} Q_3)^2 \leq C \cdot Q_3^* Q_1^{-1} Q_3$. 
This implies that $Q_3^* Q_1^{-1} Q_3 \leq C \cdot \id$, and finishes 
the proof of the lemma.
\end{proof} 
\begin{remark} 
As is apparent from this proof, $Q_2$ is 
also positive definite on $L^2(\mu)$, whereas constant 
functions belong to the kernel of $Q_1$. So it is important 
to restrict to $L^2_0 (\mu)$ for the inequality $Q_2 \leq C Q_1$
to be valid. 
\end{remark} 
We are now in a position to prove the key divergence estimate
of Cao, Lu and Wang. 
\begin{proof}[Proof of Proposition~\ref{prop_key}] 
We focus on the Sobolev estimates $i)$ and $ii)$ for now.
For these two estimates the semi-convexity hypothesis is 
not needed, only the hypothesis on the Poincar\'e constant
matters.  
We need to find $u,v$ satisfying the Dirichlet boundary conditions,
the equation $\partial_t u + L_\mu v = f$, and such that 
\begin{equation}\label{eq_goal} 
\Vert \nabla u \Vert_{L^2(\lambda \otimes \mu)}^2 + \Vert \partial_t v \Vert_{L^2(\lambda \otimes \mu)}^2 \lesssim \Vert f \Vert_{L^2(\lambda\otimes \mu)}^2 .  
\end{equation} 
Note that the choice of the 
function $v$ determines $u$. Namely $u$ is the anti-derivative 
of $f - L_\mu v = f + B v$:
\begin{equation}\label{eq_12345}
u(t,x) = \int_0^t f(s,x) + B v(s,x) \, ds . 
\end{equation} 
Then $u$ satisfies the $t=0$ boundary condition by definition, 
whereas the $t=1$ boundary condition becomes
\[ 
\int_0^1 f(t,x) + B v(t,x) \, dt = 0  .
\] 
In other words $\Pi ( f + Bv ) = 0$. 
Since $B$ commutes with $\Pi$ this amounts to 
\begin{equation}\label{eq_constraint} 
\Pi v = - B^{-1} \Pi f .  
\end{equation} 
Recall that $f$ is assumed to be centered, so 
that $\Pi f \in L^2_0 (\mu)$. Since $B^{-1}$ 
is bounded operator on $L^2_0 (\mu)$, 
the function $B^{-1} \Pi f$ is a well-defined element of 
$L^2_0 (\mu)$. 
Integrating by parts in time and space we see that 
\[ 
\Vert \nabla \partial_t v \Vert_{L^2 (\lambda\otimes \mu) }^2 
= \langle A B v , v \rangle_{ L^2 (\lambda\otimes \mu)  } .  
\] 
In a similar way, if $u$ satisfies~\eqref{eq_12345} 
and the Dirichlet boundary condition then 
\[ 
\Vert \nabla u\Vert_{L^2 (\lambda\otimes \mu)}^2 = 
\langle A^{-1} B ( B v + f ) , B v + f \rangle_{L^2 (\lambda \otimes \mu)}. 
\] 
We thus have to consider the following optimization problem: 
\begin{equation}\label{eq_opt123}
\text{minimize } \langle A^{-1} B ( B v + f ) , B v + f \rangle_{L^2 (\lambda\otimes \mu)} + \langle A B v , v\rangle_{L^2 (\lambda \otimes \mu)}, 
\end{equation} 
among functions $v$ satisfying the Dirichlet boundary condition 
as well as the constraint~\eqref{eq_constraint}. We need to 
show that the value of this optimization problem 
is at most $\Vert f\Vert^2_{L^2 (\lambda\otimes \mu)}$, 
up to a multiplicative universal constant. 
Formally we can rewrite the optimization problem as 
\begin{equation}\label{eq_problem}
\begin{cases}
\text{minimize} &  
\langle \mathcal A v ,v \rangle_{L^2(\lambda\otimes \mu)}
+ 2 \langle b ,v\rangle_{L^2(\lambda\otimes\mu)} + c  \\
\text{under} & \Pi v = d ,
\end{cases} 
\end{equation} 
where 
\begin{equation}\label{eq_abc} 
\begin{split} 
& \mathcal A = A^{-1} B ( A^2 + B^2 ) \\ 
& b = A^{-1} B^2 f\\
& c = \langle A^{-1} B f,f\rangle_{L^2 (\mu)} \\
& d = B^{-1} \Pi f .
\end{split}  
\end{equation}
This formulation is not quite legitimate. Indeed, $f$ need not
belong to the domain of $B$, so $b$ and $c$ are not really well-defined. 
We can nevertheless use this formulation to guess that the 
solution of the quadratic optimization problem should be  
\begin{equation}\label{eq_vopt}
v^{\rm opt} = \mathcal A^{-1} (- b + g) 
\end{equation} 
where $g$ is the Lagrange multiplier associated to the constraint. 
Thus $g$ must belong to the orthogonal complement of the kernel of $\Pi$,
which is also the range of $\Pi^*$. 
So there is $h\in L^2_0 (\mu)$ such that $g = \Pi^* h$.
Note that 
\[ 
\mathcal A^{-1} b = A ( A^2 + B^2)^{-1} f .
\] 
Thus the constraint $\Pi v = -B^{-1} \Pi f$ becomes
\[ 
\Pi AB^{-1} (A^2 + B^2)^{-1}\Pi^* h = 
\Pi B(A^2 + B^2)^{-1} f - B^{-1} \Pi f = - B^{-1} \Pi A^2 ( A^2 + B^2)^{-1} f .   
\] 
(recall that $B$ commutes with $\Pi$). So formally $h$ is given by 
\[ 
h = - B \left( \Pi AB (A^2+B^2)^{-1} \Pi^* \right)^{-1} \Pi A^2 ( A^2+B^2)^{-1} f.
\]  
This looks unwieldy but if we use the notations from Lemma~\ref{lem_step789}
we can rewrite this as 
\begin{equation}\label{eq_hhh}  
B^{-1} h = - Q_1^{-1} Q_2 f . 
\end{equation} 
Plugging back in~\eqref{eq_vopt} 
we get the following expression for the solution of the optimization problem 
\begin{equation}\label{eq_vopt2}
\begin{split} 
v^{\rm opt} & = - B(A^2+B^2)^{-1} f  - A(A^2+B^2)^{-1} \Pi^* Q_1^{-1} Q_3 f \\
& = - B(A^2+B^2)^{-1} f  - A^{-1} Q_3^* Q_1^{-1} Q_3 f . 
\end{split} 
\end{equation} 
Although this computation was somewhat formal, this latest expression 
defines a genuine element of $L^2 (\lambda \otimes \mu) $. 
Indeed, as we have seen before $AB(A^2+B^2)^{-1}$ is a bounded operator.
Also, when we restrict to centered functions, the operator $Q_3^* Q_1^{-1} Q_3$
is bounded, thanks to the last part of Lemma~\ref{lem_step789}. 
Therefore 
\[  
AB(A^2+B^2)^{-1} f  + Q_3^* Q_1^{-1} Q_3 f   
\] 
is a well-defined element of $L^2 (\mu)$. 
Now~\eqref{eq_vopt2} shows that $v^{\rm opt}$ is well-defined 
and belongs to the range of $A^{-1}$, which equals the domain of $A$. 
In particular $v^{\rm opt}$ satisfies the Dirichlet boundary condition. 
\\
Now we focus on the value of the optimization problem~\eqref{eq_problem}.
From~\eqref{eq_vopt} we get 
\begin{equation}\label{eq_value}
\text{value} 
= - \langle \mathcal A^{-1} b ,b \rangle_{L^2(\lambda \otimes \mu)} 
+ c + \langle \mathcal A^{-1} g , g \rangle_{L^2(\lambda\otimes \mu)} .
\end{equation}
By~\eqref{eq_abc} we have
\begin{equation}\label{eq_value1}
\begin{split}  
- \langle \mathcal A^{-1} b ,b \rangle_{L^2(\lambda \otimes \mu)} + c
& = - \langle A^{-1} B^3 (A^2+B^2)^{-1} f, f \rangle_{L^2(\lambda \otimes \mu)}
+ \langle A^{-1} B f , f \rangle_{L^2(\lambda \otimes \mu)} \\
& = \langle AB (A^2 + B^2)^{-1} f, f \rangle_{L^2 (\lambda \otimes \mu)} \\
& \leq \frac 12 \Vert f\Vert_{L^2 (\lambda \otimes \mu)}^2 . 
\end{split} 
\end{equation} 
On the other hand
\[ 
\begin{split}
\langle \mathcal A^{-1} g, g \rangle_{L^2(\lambda\otimes\mu)} 
& = \langle A B^{-1} (A^2+B^2)^{-1} \Pi^* h , \Pi^* h \rangle_{L^2 (\lambda\otimes \mu)} \\ 
& = \langle \Pi AB (A^2+B^2)^{-1} \Pi^*  B^{-1} h , B^{-1} h \rangle_{L^2(\mu)} \\
& = \langle Q_1 B^{-1} h ,  B^{-1} h \rangle_{L^2(\mu)} .  
\end{split} 
\]
Plugging in~\eqref{eq_hhh}  and using Lemma~\ref{lem_step789} again 
we thus get 
\begin{equation} \label{eq_value2}
\begin{split} 
\langle \mathcal A^{-1} g, g \rangle_{L^2(\lambda\otimes\mu)} 
& = \langle Q_1 Q_1^{-1} Q_3 f , Q^{-1} Q_3 f \rangle_{L^2 (\mu)} \\ 
& = \langle Q_3^* Q_1^{-1} Q_3 f , f \rangle_{L^2 (\lambda\otimes \mu)} \\
& \leq C \cdot \Vert f\Vert_{L^2(\lambda\times\mu)}^2. 
\end{split} 
\end{equation} 
Equations \eqref{eq_value}, \eqref{eq_value1} and \eqref{eq_value2} together 
show that the value of the optimization problem~\eqref{eq_problem} is at 
most $((1/2)+C) \Vert f\Vert_{L^2(\lambda\times\mu)}^2$. 
This proves that there 
exist $u,v$ satisfying the Dirichlet boundary condition, the divergence 
equation~\eqref{eq_laplace} and the Sobolev type estimate 
\begin{equation}\label{eq_vbnc} 
\Vert \nabla u \Vert_{L^2(\lambda\times\mu)}^2 
+ \Vert \nabla \partial_t v \Vert_{L^2(\lambda\times\mu)}^2 
\leq (\frac 12+C) \cdot \Vert f\Vert_{L^2(\lambda\times\mu)}^2 . 
\end{equation} 
It remains to prove $iii)$ and $iv)$. This is where 
the semi-convexity hypothesis and the results of 
the previous subsection enter the picture. 
Since $u,v$ satisfy the Dirichlet boundary condition we have 
\[ 
\begin{split} 
-\langle \partial_t u , L_\mu v \rangle_{L^2 (\lambda \otimes \mu)}
& = \langle u, L_\mu \partial_t v \rangle_{L^2(\lambda \otimes \mu)} \\
& = - \langle \nabla u, \nabla \partial_t v \rangle_{L^2(\lambda \otimes \mu)} \\
& \leq \frac 12 \Vert \nabla u\Vert^2_{L^2(\lambda \otimes \mu)} + \frac 12 \Vert \nabla \partial_t v\Vert^2_{L^2(\lambda \otimes \mu)}  
\end{split} 
\]
On the other hand from the equation $\partial_t u + L_\mu v = f$ we get 
\[ 
\Vert f\Vert^2_{L^2(\lambda \otimes \mu)}
= \Vert \partial_t u + L_\mu v\Vert^2_{L^2(\lambda \otimes \mu)} 
\geq \Vert L_\mu v \Vert^2_{L^2(\lambda \otimes \mu)} + 2 \langle \partial_t u ,L_\mu v \rangle_{L^2(\lambda \otimes \mu)} .  
\] 
This, together with~\eqref{eq_vbnc}, yields
\begin{equation}\label{eq_stepiiiii}
\Vert L_\mu v \Vert^2_{L^2(\lambda \otimes \mu)} 
\leq \Vert f \Vert^2_{L^2(\lambda \otimes \mu)} +  \Vert \nabla u \Vert^2_{L^2(\lambda \otimes \mu)} 
+ \Vert \nabla \partial_t v \Vert^2_{L^2(\lambda \otimes \mu)}
\leq \left( \frac 32 + C \right) \Vert f\Vert^2_{L^2(\lambda \otimes \mu) } .  
\end{equation} 
Since the potential $V$ of $\mu$ is $\kappa$ semi-convex 
and since $\mu$ satisfies Poincar\'e with constant $1$, 
we can apply Lemma~\ref{lemma_bochner1}
and Corollary~\ref{cor_bochner} from the previous section 
to the function $v(t,\cdot)$. Integrating the two inequalities 
on $[0,1]$ gives 
\begin{equation}
\label{eq_stepjjjjjj}
\Vert \nabla v \Vert^2_{L^2(\lambda\otimes\mu)} \leq \Vert L_\mu v \Vert^2_{L^2(\lambda\otimes\mu)}  
\quad \text{and} \quad  \Vert \nabla^2 v \Vert^2_{L^2(\lambda\otimes\mu)} \leq ( 1+\kappa)  \Vert L_\mu v\Vert^2_{L^2(\lambda\otimes\mu)} .  
\end{equation}
Equations~\eqref{eq_stepiiiii} and~\eqref{eq_stepjjjjjj} yield the estimates 
$(iii)$ and $(iv)$. 
\end{proof} 
\subsection{Space-time Poincar\'e inequality}
We are now in a position to prove the 
integrated Poincar\'e inequality. 
The argument is essentially the same as 
that of Cao, Lu, Wang. We spell 
it out here for completeness. 
We fix a probability measure $\mu$, and a friction 
parameter $\beta$, and we let $(P_t^*)$ be the corresponding  
kinetic Fokker-Planck semigroup. Recall that the equilibrium 
measure is $\pi :=\mu \otimes \gamma$ where $\gamma$ is the standard 
Gaussian measure.   
\begin{proposition}
Assume that the potential of $\mu$ is 
$\mathcal C^2$-smooth and semi-convex, with constant $\kappa$, 
and that $\mu$ satisfies Poincar\'e with constant $1$. 
Then for every $f \in L^2 (\pi)$ and every $T >0$ we have 
\[ 
\int_0^T \var_\pi ( P^*_t f) \, dt 
\lesssim ( T^2 + T^{-2} + \beta^2 +\kappa) \cdot \int_{[0,T]\times \R^{2n}} \vert \nabla_y P_t^*f \vert^2 \, dt d\pi.  
\] 
\end{proposition} 

\begin{proof} 
Given a function $g$ in $L^2 (\pi)$ we
denote by $M g$ its first marginal, 
namely 
\[ 
M g (x) = \int_{\R^n} g(x,y) \, \gamma (dy). 
\]  
We first note that it is enough to prove that 
\begin{equation}\label{eq_goal333}
\int_0^T \var_\mu ( M P^*_t f) \, dt 
\lesssim ( T^2 + T^{-2} + \beta^2 +\kappa) \int_{[0,T]\times \R^{2n}} \vert \nabla_y P_t^* f \vert^2 \, dt d\pi.  
\end{equation} 
Indeed we can decompose the variance and use the Gaussian Poincar\'e inequality to get   
\[
\begin{split}
\var_\pi (P_t^*f)
& = \int_{\R^n} \var_\gamma ( P_t^* f(x,\cdot) ) \, \mu (dx)  
+ \var_\mu ( M P_t^* f) \\
& \leq \int_{\R^{2n}} \vert \nabla_y P_t^* f \vert^2 \, d\pi
+ \var_\mu ( M P_t^* f) .  
\end{split}
\]
If we integrate on $[0,T]$ and combine with~\eqref{eq_goal333}
we indeed obtain the desired inequality.   
It remains to prove~\eqref{eq_goal333}. 
By a slight abuse of notation we denote 
$f(t,x,y) = P^*_tf (x,y)$. 
We can assume that $f(0,\cdot)$ is centered for $\pi$. 
This property is preserved along the diffusion, so for 
every fixed $t$, the function $f(t,\cdot)$
is centered for $\pi$, and the function 
$M f(t,\cdot)$ is centered for $\mu$. 
In particular 
\[ 
\int_{[0,T] \times \R^n} M f \, dt d\mu = 0 . 
\] 
We now invoke the previous proposition: There exist 
$u,v \in L^2 ( \lambda \otimes \mu)$ 
which vanish at $t=0$ and $t = T$ such that 
\[ 
M f = \partial_t u + L_\mu v , 
\] 
and satisfying the following Sobolev type estimates:
\begin{equation}\label{eq_sobolev}
\begin{split} 
& \Vert \nabla u \Vert_{L^2 (\lambda\otimes \mu)} \lesssim T  \Vert M f \Vert_{L^2 (\lambda \otimes \mu)} \\
& \Vert \nabla v \Vert_{L^2 (\lambda\otimes \mu)} \lesssim
 \Vert M f \Vert_{L^2 (\lambda\otimes \mu)} \\
& \Vert \nabla \partial_t v\Vert_{L^2 (\lambda\otimes \mu)}
\lesssim T^{-1}  \Vert M f \Vert_{L^2 (\lambda\otimes \mu)} \\
& \Vert \nabla^2 v \Vert_{L^2 (\lambda\otimes \mu)}
\lesssim (1+\kappa )^{1/2}  \Vert M f \Vert_{L^2 (\lambda\otimes \mu)}
\end{split} 
\end{equation}
Notice that as $v$ is a function depending only on $t$ and $x$ but not on the $y$ variable, we have  
\[ 
\mathcal L v(t,x) = \langle \nabla_x v(t,x) , y\rangle  
\] 
and also 
\[ 
\mathcal L^2 v(t,x) = \langle \nabla^2 v(t,x) y , y \rangle 
- \beta \langle \nabla v(t,x) , y \rangle - \langle \nabla v(t,x) , \nabla V(x) \rangle. 
\] 
Integrating in the $y$ variable and using the fact that the standard Gaussian 
has mean $0$ and identity covariance matrix we get 
\begin{equation}\label{eq111}
M \mathcal L^2 v= L_{\mu} v . 
\end{equation}
Therefore 
\[ 
\begin{split} 
\Vert M f \Vert^2_{L^2 ( \lambda \otimes \mu )} 
& = \langle M f , \partial_t u + L_\mu v \rangle_{L^2 ( \lambda \otimes \mu )} \\ 
& = \langle M f , \partial_t u  +  M \mathcal L^2 v \rangle_{L^2 ( \lambda \otimes \mu )} \\ 
& = \langle f ,  \partial_t u  \rangle_{L^2 ( \lambda \otimes \pi )} 
+ \langle M f , \mathcal L^2 v \rangle_{L^2 ( \lambda \otimes \pi )} \\
& = \langle f ,  \partial_t u + \mathcal L^2 v \rangle_{L^2 ( \lambda \otimes \pi )} - \langle f-M f ,  \mathcal L^2 v \rangle_{L^2 ( \lambda \otimes M )} . 
\end{split} 
\] 
Moreover, since $u$ vanishes on the boundary of $[0,T]\times \R^n$
and since $f$ satifies the kinetic Fokker-Planck equation, we have
\[ 
\begin{split} 
\langle
f , \partial_t u \rangle_{L^2 (\lambda \otimes \pi )} 
& = - \langle \partial_t f , u  \rangle_{L^2 (\lambda \otimes \pi )} \\
& =- \langle \mathcal L^* f,  u \rangle_{L^2 (\lambda \otimes \pi )} \\
& =- \langle f , \mathcal L u \rangle_{L^2 (\lambda \otimes \pi )}. 
\end{split}  
\]  
In a similar way, since $v(t,x)$ vanishes
when $t= 0$ and $t=T$, the function $\mathcal L v (t,x,y) = 
\langle \nabla v(t,x) ,y \rangle$ also has this property. 
As a result 
\[ 
\begin{split} 
\langle f , \mathcal L^2 v \rangle_{L^2(\lambda \otimes \pi)} 
& = \langle \mathcal L^* f , \mathcal L v  \rangle_{L^2(\lambda \otimes \pi)} \\
& =  \langle \partial_t f , \mathcal L v  \rangle_{L^2(\lambda \otimes \pi)}  \\
& = - \langle f ,\mathcal L \partial_t v \rangle_{L^2(\lambda \otimes \pi)}  . 
\end{split}
\]  
Therefore 
\[ 
\Vert M f \Vert^2_{L^2(\lambda \otimes \pi)}  
= - \langle  f , \mathcal Lu + \mathcal L\partial_t v \rangle_{L^2(\lambda \otimes \pi)} 
-  \langle f- M f ,  \mathcal L^2 v \rangle_{L^2(\lambda \otimes \pi)} .
\] 
Next we replace $\mathcal L^2 v$ by its expression and we observe that 
some terms cancel out. We finally obtain 
\begin{equation}\label{eq_333}
\begin{split}  
\Vert M f \Vert^2_{L^2(\lambda \otimes \pi)} 
= & - \langle f , \mathcal L  u + \mathcal L \partial_t v + \beta \mathcal L v \rangle_{L^2(\lambda \otimes \pi)} \\
& - \langle f- M f,  h \rangle_{L^2(\lambda \otimes \pi)} , 
\end{split} 
\end{equation} 
where $h$ is the function given by 
\begin{equation}\label{eq_defh}
h(t,x,y) = \langle \nabla_x^2 v(t,x) y , y \rangle . 
\end{equation} 
Now we will start writing inequalities. Recall that 
if $w$ is a function not depending on $y$ then 
$\mathcal L w = \langle \nabla_x w , y \rangle$. 
Together with the Gaussian integration by parts formula:
\[ 
\int_{\R^n} f(x,y) y\, \gamma(dy) = \int_{\R^n} \nabla_y f (x,y) \, \gamma (dy) 
\] 
and Cauchy-Schwarz we get 
\[ 
\langle f , \mathcal Lw\rangle_{L^2(\lambda \otimes \pi)} 
= \langle \nabla_y f , \nabla_x w \rangle_{L^2(\lambda \otimes \pi)} 
\leq \Vert \nabla_y f \Vert_{L^2 (\lambda \otimes \pi)} \cdot \Vert \nabla_x w \Vert_{L^2 (\lambda \otimes \mu ) } . 
\] 
This allows to upper bound the first term of the 
right-hand side of~\eqref{eq_333}. For the 
second term we use the Gaussian Poincar\'e inequality
again. Note that if $A$ is a 
fixed symmetric matrix then the gradient of $\langle Ay,y\rangle$ is $2 Ay$, 
and thus 
\[ 
\var_\gamma ( \langle Ay,y\rangle ) \leq 4 \int_{\R^n} \vert Ay\vert^2 \, d\gamma = 4 \Vert A \Vert^2_F , 
\]   
(actually if we were tracking down constants we could save a factor $2$ here).
As a result, using Cauchy-Schwarz and the definition~\eqref{eq_defh} 
of the function $h$ we get  
\[ 
\begin{split} 
\langle f- M f,  h \rangle_{L^2(\lambda \otimes \pi)} 
& \leq \int_{[0,T]\times \R^n} \var_\gamma (f(t,x,\cdot))^{1/2} \cdot \var_\gamma ( h(t,x,\cdot) )^{1/2} \, dt d\mu \\ 
& \leq \int_{[0,T]\times \R^n} \Vert \nabla_y f \Vert_{L^2 (\gamma)} \cdot 
2 \Vert \nabla_x^2 v \Vert_{F} \, dt d\mu \\
& \leq 2 \Vert \nabla_y f \Vert_{L^2 (\lambda \otimes \pi)} \cdot \Vert \nabla^2_x v\Vert_{L^2 (\lambda\otimes \mu)} . 
\end{split} 
\] 
Putting everything together we obtain 
\[ 
\Vert M f \Vert^2_{L^2(\lambda\otimes \mu)}
\leq D \cdot \Vert \nabla_y f \Vert_{L^2 (\lambda \otimes \pi)}
\] 
with 
\[ 
D = 
\Vert \nabla_x  u  \Vert_{L^2( \lambda \otimes \mu ) } 
+ \Vert \nabla_x \partial_t v \Vert_{L^2 ( \lambda \otimes \mu)} 
+ \beta \Vert \nabla_x v \Vert_{L^2 (\lambda \otimes \mu)}
+ 2 \Vert \nabla_x^2 v \Vert_{L^2 ( \lambda \otimes \mu )}  . 
\] 
Plugging in the bounds~\eqref{eq_sobolev} we finally get
\[ 
\Vert M f \Vert_{L^2( \lambda \otimes \mu)}
\lesssim ( T+T^{-1} + \beta + \sqrt{1+\kappa} )
\Vert \nabla_y f \Vert_{L^2 ( \lambda \otimes \pi)} . 
\] 
This implies~\eqref{eq_goal333}
and finishes the proof of the proposition. 
\end{proof} 
\subsection{Proof of the hypocoercive estimate} 
We can now finally prove the hypocoercive estimate, Theorem~\ref{thm_hypo2}.
Again that part of the argument is pretty much the same as in the Cao, Lu, Wang 
paper and we include it for completeness. 

As we already mentioned there is some homogeneity 
in the problem which allows us to rescale the measure 
$\mu$. In particular it is enough to prove the result 
assuming $C_P=1$, say. By the previous proposition, given 
$f \in L^2 (\pi)$, we then have 
\[  
\int_0^T \var_\pi ( P^*_t f) \, dt 
\lesssim \left( T^2 + \frac 1{T^2} + \beta^2+ \kappa \right) 
\int_{[0,T]\times \R^{2n}} \vert \nabla_y P^*_tf \vert^2 \, dt d\pi . 
\]
Since 
\[ 
\frac d{dt} \var_\pi (P^*_tf)
= - 2\beta \int_{\R^{2n}} \vert \nabla_y P_t^* f \vert^2 \, d\pi \leq 0 ,
\] 
the previous inequality implies that 
\[
\begin{split}   
T \cdot \var_\pi ( P_T^* f ) & \leq \int_0^T \var_\pi ( P^*_t f) \, dt \\
& \lesssim \frac 1\beta \left( T^2 + \frac 1{T^2} + \beta^2 + \kappa \right) 
\left( \var_\pi ( f) - \var_\pi (P^*_T f) \right) . 
\end{split} 
\] 
Hence 
\[ 
\var_\pi ( P_T^* f ) \leq \frac{ 1 }{ 1 + x(T) } \, \var_\pi (f) ,  
\] 
where
\[ 
x(T) = \frac{ \beta T}{ C( T^2 + \frac 1{T^2} + \beta^2 + \kappa) }.  
\] 
and where $C$ is a universal constant, which we can assume 
to be larger than $1$. Since $\beta^2 + T^2 \geq 2 \beta T$, 
we have $x(T) \leq 1/(2C) \leq 1/2$, no matter what the value of $T$
and $\beta$. For $x\leq 1/2$ 
we certainly have $1/(1+x) \leq \e^{-x/2}$, and thus 
\[
\var_\pi (P^*_T f) \leq \e^{- x(T)/2 } \var_\pi ( f) ,  
\]
for all $T>0$. Applying this inequality to $P^*_Tf$, then $P^*_{2T} f$, and 
so on, and using the semigroup property we get 
\[ 
\var_\pi (P^*_{nT} f) \leq \e^{- n \cdot x(T) /2 } \var_\pi ( f) ,  
\] 
for every $T>0$ and every integer $n$. Equivalently
\[ 
\var_\pi (P^*_t f) \leq \e^{-  n \cdot x(t/n) /2 } \var_\pi ( f) ,  
\] 
for every $t>0$ and every integer $n\neq 0$. It remains to optimize on $n$. 
Since 
\[ 
n \cdot x ( \frac tn )  = \frac { t \beta } { C ( \frac{t^2}{n^2} + \frac{n^2}{t^2} +\beta^2 + \kappa )} ,  
\] 
assuming $t\geq 1$ and choosing $n$ to be the integer part of $t$ 
yields 
\[ 
\var_\pi (P^*_t f) \leq \exp \left( -c  \cdot \frac{t\beta}{1+\beta^2 +\kappa} \right)  \var_\pi (f) , \quad \forall t \geq 1  
\] 
and where $c$ is a small universal constant. 
For $t\leq 1$ we can just use the trivial bound 
$\var_\pi ( P^*_t f ) \leq \var_\pi (f)$, so that the
last inequality implies  
\[ 
\var_\pi (P^*_t f) \leq 2 \cdot \exp \left( -c'' \cdot \frac{t\beta }{1+\beta^2 +\kappa} \right)  \var_\pi (f) , \quad \forall t \geq 0,    
\] 
which is the result.  
\begin{remark} 
The proof also yields a non trivial 
estimate for small values of $t$, namely 
\[ 
\var_\pi ( P^*_t f ) \leq \exp \left( - c \cdot \frac{t \beta} {1+ t^{-2} +\beta^2 + \kappa } \right) \var_\pi (f) . 
\] 
\end{remark} 

\section{The discretization argument} \label{sec_discr}
It remains to estimate the discretization error.
The approach is taken from Dalalyan~\cite{dalalyan1} 
and relies on the Girsanov change of measure formula. 
The relative entropy (a.k.a. Kullback divergence) plays an 
important role in this approach. Recall its definition: If 
$\mu,\nu$ are probability measures defined on the same 
space $X$ then 
\[ 
D ( \nu \mid \mu) = \int_X \log \left( \frac{d\nu}{d\mu} \right) \, d\nu , 
\] 
assuming that $\nu$ is absolutely continuous with respect 
to $\mu$. If this is not the case the convention is that 
the relative entropy equals $+\infty$. An important 
feature of relative entropy is that it controls the total 
variation distance: 
\[ 
TV ( \mu , \nu ) \leq \sqrt{ \frac 12 D ( \nu \mid \mu)  } . 
\] 
This is known as Pinsker's inequality. 
Another property that we will need is that taking a marginal 
can only lower the relative entropy. More generally, if $T \colon X\to Y$ 
is any measurable map then 
\begin{equation}\label{eq_contraction}
D ( T \# \nu \mid T\#\mu ) \leq D ( \nu \mid \mu ) .   
\end{equation}
Here $T \# \mu$ denotes the pushforward of the 
measure $\mu$ by the map $T$.  
Note that this contraction property is not
specific to the Kullback divergence, the 
chi-square divergence also has this property. 
Lastly we will need the following elementary 
lemma, which can be thought of as an approximate 
triangle inequality for the relative entropy.   
\begin{lemma} \label{lem_triangle} 
Let $\mu_1,\mu_2,\mu_3$ be probability 
measures defined on the same space $E$.
Then 
\[ 
D ( \mu_3 \mid \mu_1 ) \leq 2 D( \mu_3 \mid \mu_2 ) 
+ \log \left ( 1 + \chi_2 ( \mu_2 \mid \mu_1 ) \right) . 
\] 
\end{lemma} 
\begin{proof}
We can assume that $\mu_3 \prec \mu_2 \prec \mu_1$, where the symbol 
$\prec$ denotes absolute continuity. Indeed, if $\mu_3$ is 
not absolute continuous with respect to $\mu_2$ or if $\mu_2$ is 
not absolutely continuous with respect to to $\mu_1$ then the 
inequality trivially holds true. 
We can also assume that $\mu_1 \prec \mu_2 \prec \mu_3$ 
(and thus that all three measures are equivalent) by some limiting argument. 
Then  
\begin{equation} \label{eq_pml}
\begin{split} 
D ( \mu_3 \mid \mu_1 ) 
& = \int_E \log \left( \frac{d\mu_3}{d\mu_1} \right) \, d\mu_3 \\
& = D ( \mu_3 \mid \mu_2 ) + \int_E \log \left( \frac{d\mu_2}{d\mu_1} \right) \, d\mu_3 .
\end{split}  
\end{equation}
Moreover, by Jensen's inequality we have
\[
\begin{split}  
\int_E \log f \, d\mu_3 
& = \int_E  \log  \left( f \cdot \frac{ d\mu_1 }{d\mu_3} \right)  \, d\mu_3 + D(\mu_3 \mid \mu_1 )  \\
& \leq \log \left( \int_E f \, d\mu_1 \right) + D ( \mu_3 \mid \mu_1 ) , 
\end{split} 
\]  
for any positive function $f$. Applying this $f = (d\mu_2/d\mu_1)^2$ we get 
\[ 
2 \int_X \log \left( \frac{d\mu_2}{d\mu_1} \right) \, d\mu_3 
\leq \log \left( 1 + \chi_2 ( \mu_2 \mid \mu_1 ) \right) 
+ D( \mu_3 \mid \mu_1 ) .  
\] 
Plugging this back into~\eqref{eq_pml} yields the result. 
\end{proof} 
%
\begin{proposition} 
Assume that $V$ is gradient Lipshitz, with Lipschitz constant $L$. 
Given a probability measure $\nu$ on $\R^{2n}$ 
let $\nu_t$ be the law of the discretized kinetic Langevin diffusion~\eqref{eq_euler2} starting from $\nu$ at time $t$, 
and recall that $\nu P_t$ denotes the law of the true diffusion at 
time $t$. Then 
\[ 
TV (\nu_t , \nu P_t ) 
\lesssim \frac{ \sqrt t \cdot L \eta}{\sqrt{\beta}} 
\cdot ( \sqrt{n} + \sqrt{ \log (1+\chi_2 ( \nu \mid \pi) ) } ) , 
\] 
where $\eta$ is the time step and $\beta$ is the friction parameter. 
\end{proposition} 
\begin{proof} 
We use the following version of Girsanov: 
Let $(W_t)$ be a standard Brownian motion on $\R^n$, 
defined on some probability space $(\Omega ,\mathcal F, \mathbb P)$ 
equipped with some filtration $(\mathcal F_t)$. 
Let $(X_t)$ be a process of the form $X_t = W_t + \int_0^t u_s \, ds$ 
where $(u_t)$ is a progressively measurable process satisfying 
some integrability conditions. Fix a time 
horizon $T$. The new measure $\mathbb Q$ defined by  
\begin{equation}\label{eq_girsanov}
\frac{d \mathbb Q}{d \mathbb P} =
 \exp \left( - \int_0^t u_s dW_s -\frac 12 \int_0^t \vert u_s\vert^2 ds\right)  
\end{equation}
is a probability measure under which the process $(X_t)_{t\leq T}$ 
is a standard Brownian motion. 
Now consider the solution $(X^\eta,Y^\eta)$ of the discretized Langevin equation~\eqref{eq_euler2} initiated at some measure $\nu$, in the sense that the starting point $(X^\eta_0,Y^\eta_0)$ has law $\nu$, and consider the process $(\tilde W_t)$ given by  
\[
\tilde W_t = W_t + \frac 1{\sqrt{2\beta}} 
\int_0^t 
\left( \nabla V ( X^\eta_s ) - \nabla V( X^\eta_{\lfloor s/\eta \rfloor \eta} ) \right) \, ds . 
\] 
Observe that by definition of $\tilde W$ we have 
\[ 
\begin{cases} 
d X^\eta_t = Y^\eta_t \, dt \\
d Y^\eta_t = \sqrt{2\beta} \, d \tilde W_t - \beta Y^\eta_t \, dt - \nabla V(X^\eta_t) \, dt. 
\end{cases} 
\] 
In other words, if we replace $W$ by $\tilde W$ then $(X^\eta_t,Y^\eta_t)$ 
becomes a genuine kinetic Langevin diffusion. According to Girsanov, if 
we set $u_s = (2\beta)^{-1/2} (\nabla V ( X^\eta_s ) - \nabla V( X^\eta_{\lfloor s/\eta\rfloor  \eta} )$ 
and define $\mathbb Q$ by~\eqref{eq_girsanov}, then $\tilde W$ is 
standard Brownian motion under $\mathbb Q$, so that $(X^\eta,Y^\eta)$ is a genuine kinetic 
Langevin diffusion under $\mathbb Q$. Therefore 
\[ 
(X^\eta_t,Y^\eta_t) \# \mathbb Q = \nu P_t \quad \text{and}\quad
(X^\eta_t,Y^\eta_t) \# \mathbb P = \nu_t .  
\] 
Hence from~\eqref{eq_contraction} 
\[ 
D( \nu_t \mid \nu P_t ) 
\leq D ( \mathbb P \mid \mathbb Q ) . 
\] 
Moreover since $\int u_s \cdot dW_s$ is a martingale 
having expectation $0$, 
\[ 
 D ( \mathbb P \mid \mathbb Q ) = \E \log \frac{d\mathbb P}{d\mathbb Q} 
  = \frac1{4\beta} \int_0^t \E \vert \nabla V (X^\eta_s) -
 \nabla V( X^\eta_{\lfloor s/\eta \rfloor \eta} ) \vert^2 \, ds . 
\] 
Hence the estimate
\begin{equation}\label{eq_pdrt}
D( \nu_t \mid \nu P_t ) 
\leq \frac1{4\beta} \int_0^t \E \vert \nabla V (X^\eta_s) -
 \nabla V( X^\eta_{\lfloor s/\eta\rfloor \eta} ) \vert^2 \, ds . 
\end{equation} 
Strictly speaking the Girsanov change of measure formula 
is only valid under some integrability condition on the drift $(u_t)$. 
For instance $(u_t)$ uniformly bounded on $[0,T]$ is enough. 
However using some
localization argument one can get around this issue and show that 
inequality~\eqref{eq_pdrt} remains valid under no further assumption. 
See~\cite{follmer} for more details. Suppose 
that $t$ is an integer multiple of $\eta$. Fix an integer $k\leq N-1$, 
where $N = t / \eta$. Since $\nabla V$ is Lipschitz with constant $L$, for every $s\in [k\eta , (k+1) \eta]$
we get from Cauchy-Schwarz  
\[ 
\E \vert \nabla V (X^\eta_s) -
\nabla V( X^\eta_{k\eta} ) \vert^2 
\leq L^2 \E  \vert  X^\eta_s-X^\eta_{k\eta} \vert^2 
\leq L^2 \eta \int_{k \eta}^s \E \vert Y^\eta_u \vert^2 \, du  .
\] 
 Integrating this
inequality on the interval 
$s\in [k\eta , (k+1) \eta]$, 
summing over $k\leq N-1$, and plugging 
back in~\eqref{eq_pdrt} we finally get 
\begin{equation} \label{eq_444}
D( \nu_t \mid \nu P_t ) 
\leq \frac{L^2 \eta^2 }{4\beta} \int_0^t \E \vert Y^\eta_s\vert^2 \, ds . 
\end{equation} 
It remains to control $\E \vert Y^\eta_s\vert^2$. Intuitively 
the law of $Y^\eta_s$ should not be too far away from that of the 
second factor of the equilibrium measure, which is the 
standard Gaussian $\gamma$. 
Therefore we should have $\E \vert Y^\eta_s\vert^2 \lesssim n$. 
Let $\nu_{s,2}$ be
the second marginal of $\nu_s$, namely the law of $Y^\eta_s$. 
Observe that if $Y,G$ is a coupling of $(\nu_{s,2},\gamma)$, 
then  
\[ 
\E \vert Y^\eta_s \vert^2 = \E \vert Y\vert^2 
\leq 2 \E \vert Y_s - G \vert^2 + 2 \E \vert G\vert^2 = 
 2 \E \vert Y_s - G \vert^2 + 2 n. 
\] 
Taking the infimum over such couplings yields 
\[ 
\E \vert Y^\eta_s \vert^2 \leq  2 T_2 ( \nu_{s,2}, \gamma ) + 2 n ,
\] 
where $T_2$ is the transportation distance from $\nu_{2,s}$ to $\gamma$ 
associated to the cost function $\vert x-y\vert^2$. Next we invoke 
Talagrand's inequality~\cite{talagrand}:
\[ 
T_2 ( \nu_{s,2}, \gamma ) \leq 2 D (\nu_{s,2} \mid \gamma )   .
\] 
Combining this with the contraction principle and Lemma~\ref{lem_triangle} 
we get
\[
\begin{split}  
\E \vert Y^\eta_s\vert^2
& \leq 2 n + 4 D( \nu_{s,2} \mid \gamma )  \\
& \leq 2 n + 4 D ( \nu_s \mid \pi) \\
& \leq 2 n + 8 D(\nu_{s} \mid \nu P_s )
+ 4 \log \left( 1 + \chi_2 ( \nu P_s \mid \pi ) \right) . 
\end{split} 
\]
Recall that $\chi_2 ( \nu P_s \mid \pi )$ is non increasing with 
$s$ (see the previous section). Together with~\eqref{eq_444} we 
thus obtain 
\[ 
\E \vert Y^\eta_t\vert^2 \leq  2 n + 4 \log \left( 1 + \chi_2 ( \nu \mid \pi )  \right) + \frac{2 L^2 \eta^2}{\beta} \int_0^t \E \vert Y^\eta_s\vert^2 \, ds , 
\] 
for all $t>0$. Using Gronwall's lemma and the convexity 
on the exponential function we see that this implies 
\[ 
\int_0^t \E \vert Y^\eta_s\vert^2 \, ds  \lesssim t \cdot  ( n + \log \left( 1 + \chi_2 ( \nu \mid \pi )  \right)   ) 
\] 
as long as $2 t L^2 \eta^2 \leq \beta$. 
Plugging this back into~\eqref{eq_444} yields
\[ 
D(\nu_t \mid \nu P_t ) \lesssim \frac{t L^2 \eta^2}{\beta} \cdot ( n + \log 
( 1 + \chi_2 ( \nu \mid \pi ) )   ) .  
\] 
Combining this with Pinsker's inequality, 
we get the result under the additional hypothesis that 
$2t L^2 \eta^2 \leq \beta$. However since the total 
variation is anyways less than $1$, if this additional hypothesis 
is violated then the result trivially holds true. 
\end{proof} 

\section{Proof of the main result} 
In this section we wrap up the proof of Theorem~\ref{thm_main}. 

Let $(x_k,y_k)$ be the Markov chain induced by the 
discretization with  time step $\eta$ of the kinetic Langevin diffusion
and assume that the law of the initial point is a product 
measure whose second factor is $\gamma$. We write 
\[ 
TV ( x_k , \mu )  \leq TV ( (x_y,y_k) , \pi )
\leq TV ( (x_k,y_k) , (X_t,Y_t) ) + TV ( (X_t,Y_t) , \pi ) , 
\] 
where $t = k \eta$. 
Recall that the notation $TV ((x,y),\pi)$  
stands for the total variation between the law of $(x,y)$ and $\pi$. 
Also, as the reader probably guessed already, $(X_t,Y_t)$ denotes
the genuine kinetic Langevin diffusion initiated at the same 
point as the algorithm. To upper bound the first term we 
invoke the latest proposition. For the second term we 
use the bound
\[  
TV ( (X_t,Y_t) , \pi ) \leq \sqrt{ \chi_2 ( (X_t,Y_t) \mid \pi ) } 
\] 
and we apply Theorem~\ref{thm_hypo}. Since 
\[ 
\chi_2 ( (X_0,Y_0) \mid \pi ) = \chi_2 ( (x_0,y_0) \mid \pi ) = \chi_2 ( x_0 \mid \mu) ,  
\] 
we finally arrive at 
\begin{equation}\label{eq_6666}
TV ( x_k , \mu )  
\lesssim  \frac{L \eta \sqrt{t}}{\sqrt \beta} \cdot (\sqrt n + \sqrt{ \log (1 +\chi_2 (x_0 \mid \mu ) } ) 
+ \exp \left( - \frac{ c t \beta }{ 1+(\beta^2+\kappa) C_P}   \right) \sqrt{\chi_2 (x_0\mid \mu) } .  
\end{equation} 
We still have the freedom to optimize on $\beta$. 
Let us focus on the log-concave case for now. Namely 
we assume that $\kappa =0$. In this situation the best choice  
is $\beta = C_p^{-1/2}$ in which case the latest displays becomes  
\[
TV ( x_k \mid \mu )  
\lesssim  L \eta C_P^{1/4} t^{1/2} \cdot (\sqrt n + \sqrt{ \log (1 +\chi_2 (x_0 \mid \mu ) } ) 
+ \exp \left( - \frac{ c t}{ \sqrt{C_P}}   \right) \sqrt{\chi_2 (x_0\mid \mu) } 
\]
Fix $\epsilon \in (0,1)$. 
For the second term to be of order $\epsilon$ we need to 
take 
\[ 
t\approx \sqrt{C_P} \cdot  \log ( \frac{\chi_2(x_0\mid\mu) }\epsilon ) .   
\] 
Then we adjust $\eta$ so that the first term equals $\epsilon$, 
and the corresponding number of steps is 
\[
\begin{split}  
k = \frac t \eta 
& \approx \epsilon^{-1} C_P^{1/4} L t^{3/2} \cdot 
( \sqrt n + \sqrt{ \log (1 + \chi_2 (x_0 \mid\mu) ) }  ) \\
& \approx \epsilon^{-1} L C_P \cdot C(n,\epsilon,x_0) 
\end{split} 
\]
where 
\[ 
C (n,\epsilon ,x_0) = ( \sqrt n + \sqrt{  \log ( 1 + \chi_2 ( x_0\mid \mu) } ) \cdot \log \left(\frac{ \chi_2 ( x_0 \mid \mu) } \epsilon \right)^{3/2} . 
\]  
This finishes the proof of the main result in the log-concave case. 
For the general case, observe that being gradient Lipschitz is
a stronger assumption than being semi-convex. We can thus replace 
$\kappa$ by $L$ in~\eqref{eq_6666}. 
Note that we always have $L C_P \geq 1$. This follows from 
Caffarelli's contraction theorem~\cite{caffarelli}. 
Therefore 
\[ 
\exp \left( - \frac{ c t \beta }{ 1+(\beta^2+L)C_P} \right)
\leq \exp \left( - \frac{ c' t \beta }{ (\beta^2 +L)C_P}  \right) . 
\]
Choosing $\beta = \sqrt L$ then yields 
\[ 
TV ( x_k , \mu )  
\lesssim   \eta L^{3/4} t^{1/2} \cdot (\sqrt n + \sqrt{ \log (1 +\chi_2 (x_0 \mid \mu ) } ) 
+ \exp \left( - \frac{ c'' t}{ \sqrt{L} C_P} \right) \sqrt{\chi_2 (x_0\mid \mu) } .
\] 
We then conclude as in the log-concave case.

\bibliographystyle{plain}
\bibliography{biblio}

\end{document}